\documentclass[10pt]{article}
\usepackage{amsmath,amssymb,amsthm,amsfonts,amsbsy,euscript,amscd,epsfig,epigraph}
\usepackage{graphics,graphicx,epstopdf,multicol,enumerate,color,cite}
\usepackage[lmargin=1.in,rmargin=1.in,tmargin=1.in,bmargin=1.in]{geometry}

\newtheorem{theorem}{Theorem}
\newtheorem{corollary}{Corollary}
\newtheorem{remark}{Remark}

\newcommand{\ds}{\displaystyle}
\newcommand{\s}{\sigma}
\newcommand{\R}{\mathbb{R}}
\newcommand{\C}{\mathbb{C}}
\newcommand{\bbS}{\mathbb{S}}
\newcommand{\bb}{{\bf b}}
\newcommand{\bc}{{\bf c}}

\newcommand{\bff}{{\bf f}}
\newcommand{\bg}{{\bf g}}

\newcommand{\bp}{{\bf p}}
\newcommand{\bu}{{\bf u}}
\newcommand{\bv}{{\bf v}}
\newcommand{\bx}{{\bf x}}
\newcommand{\by}{{\bf y}}

\newcommand{\bA}{{\bf A}}
\newcommand{\bB}{{\bf B}}
\newcommand{\bC}{{\bf C}}
\newcommand{\bE}{{\bf E}}

\newcommand{\bH}{{\bf H}}
\newcommand{\cbH}{\mbox{\boldmath${\EuScript{H}}$} }
\newcommand{\cbJ}{\mbox{\boldmath${\EuScript{J}}$} }
\newcommand{\bI}{{\bf I}}

\newcommand{\bN}{{\bf N}}
\newcommand{\bU}{{\bf U}}
\newcommand{\bV}{{\bf V}}
\newcommand{\bW}{{\bf W}}
\newcommand{\bY}{{\bf Y}}
\newcommand{\bbr}{{\bf \widetilde{b}}}
\newcommand{\bcr}{{\bf \widetilde{c}}}
\newcommand{\bffr}{{\bf \widetilde{f}}}
\newcommand{\bgr}{{\bf \widetilde{g}}}
\newcommand{\bxr}{{\bf \widetilde{x}}}
\newcommand{\byr}{{\bf \widetilde{y}}}
\newcommand{\bhp}{{\bf \widehat{p}}}

\newcommand{\bAr}{{\bf \widetilde{A}}}
\newcommand{\bBr}{{\bf \widetilde{B}}}
\newcommand{\bCr}{{\bf \widetilde{C}}}
\newcommand{\bEr}{{\bf \widetilde{E}}}

\newcommand{\bHr}{{\bf \widetilde{H}}}
\newcommand{\bNr}{{\bf \widetilde{N}}}
\newcommand{\bYr}{{\bf \widetilde{Y}}}
\newcommand{\cA}{\mathcal{A}}

\newcommand{\cP}{\mathcal{P}}
\newcommand{\cQ}{\mathcal{Q}}
\newcommand{\cAr}{\mathcal{\widetilde{A}}}

\begin{document}

\title{Interpolatory Model Reduction of Parameterized Bilinear Dynamical Systems
	\thanks{
		Supported in part by the National Science Foundation under contract DMS--1522616 and the National Institute
		for Occupational Safety and Health under contract 200-2014-59669. The work of Gugercin was also supported in part by 
		the Alexander von Humboldt Foundation.
	}}
\date{}
\author{Andrea Carracedo Rodriguez\footnote{Department of Mathematics, Virginia Tech, Blacksburg, VA, USA. \newline \texttt{\{crandrea, gugercin, jborggaard\}@vt.edu}} 
\and Serkan Gugercin$^\dag$ 
\and Jeff Borggaard$^\dag$}

\maketitle

\vspace{-4ex}
\begin{center}
\textit{Originally submitted for publication in July 2017}
\end{center}

\begin{abstract}
Interpolatory projection methods for model reduction of nonparametric linear dynamical systems have been successfully extended to nonparametric bilinear dynamical systems. 
However, this is not the case for parametric bilinear systems. In this work, we aim to close this gap by providing a natural extension of interpolatory projections to model reduction of parametric bilinear dynamical systems.
We introduce necessary conditions that the projection subspaces must satisfy to obtain parametric tangential interpolation of each subsystem transfer function. 	
These conditions also guarantee that the parameter sensitivities (Jacobian) of each subsystem transfer function is matched tangentially by those of the corresponding reduced order model transfer function. 
Similarly, we obtain conditions for interpolating  the parameter Hessian of  the transfer function by including extra vectors in the projection subspaces.
As in the parametric linear case, the basis construction for two-sided projections does not require computing the Jacobian or the Hessian.
\end{abstract}


\section{Introduction}
	\label{intro}


Simulation of dynamical systems has become an essential part in the development of science to study complex physical phenomena. 
However, as the ever increasing need for accuracy has lead to ever larger dimensional dynamical systems, this increased dimension often makes the desired numerical simulations prohibitively expensive to perform. 
Model order reduction (MOR) is one remedy for this predicament. 
MOR tackles this issue by constructing a much lower dimensional representation of the corresponding full-order dynamical system, which is cheap to simulate, yet provides high-fidelity, i.e., it provides a good approximation to the original quantity of  interest. 
In many applications such as optimization, design, control, uncertainty quantification, and inverse problems, the dynamics of the system are defined by a set of parameters that describe initial conditions, material properties, etc.  
Since carrying out model reduction for every parameter value is not computationally feasible, the goal in the parameterized setting is to construct a parametric reduced model that can approximate one or more quantities of interest well for the whole parameter range of interest.  
This lead to the parametric model reduction framework. For more specific details on both parametric and nonparametric model reduction, we refer the reader to \cite{antoulas2001asurvey,baur2014model,antoulas2005approximation,benner2015survey,benner2017model,hesthaven2016certified} and the references therein. 

In this paper, we will focus on large-scale bilinear systems parametrized with the parameter vector $\bp \in \R^\nu$ and represented in state-space form
	\begin{align} 
	\left\{ \begin{array}{l}
	\label{binon}
	\bE ( \bp ) \dot{ \bx } (t;\bp) 	 = \ds \bA (\bp) \bx (t;\bp) + \sum_{j=1}^m \bN_j (\bp) \bx (t) u_j (t) + \bB (\bp) \bu (t), \\[1ex]
	\by (t;\bp) 		 = \bC (\bp) \bx (t;\bp), 		
	\end{array} \right.
	\end{align}
where  $\bx(t;\bp) \in \R^n$,  $\by(t;\bp) \in \R^\ell$, and $\bu(t) = [u_1(t),~u_2(t),\ldots,~u_m(t)]^\top\in \R^m$ denote the states, outputs (measurements/quantities of interest), and inputs (excitation/forcing) of the bilinear dynamical system, respectively. Thus, the corresponding state-matrices have the dimensions  
$\bE(\bp),\bA(\bp), \bN_j(\bp) \in \R^{n\times n}$, for $j=1,\dots m$, $\bB(\bp) \in \R^{n\times m}$, and $\bC(\bp) \in \R^{\ell\times n}$. In this paper, we assume that the matrix $\bE(\bp)$ is nonsingular for every parameter value $\bp \in \R^\nu$.
Bilinear systems of the form \eqref{binon} appear in a variety of applications 
such as the study of biological species and nuclear fission, 
are used in the context of stochastic control problems,
and frequently appear in modeling nonlinear phenomena of small magnitude, for instance, \cite{rugh1981nonlinear,mohler1991nonlinear,mohler1970natural,weiner1980sinusoidal,hartmann2013balanced,benner2011lyapunov,benner2017dual}.   
We are interested in large-scale settings where simulating/solving \eqref{binon} for a wide variety of inputs $ \bu (t) $  and parameters $\bp$ solely to determine the output $ \by (t;\bp) $ is too expensive. Therefore, our goal is to construct a reduced parametric bilinear system of order $r \ll n$ in state-space form
	\begin{align}
	\widetilde{\Sigma} : \ \left\{ 
	\begin{array}{l}
	\label{brom}
	\bEr ( \bp ) \dot{ \bxr } (t;\bp) 	 = \ds \bAr ( \bp ) \bxr (t;\bp) + \sum_{j=1}^m \bNr_j ( \bp ) \bxr (t;\bp) u_j (t) + \bBr ( \bp ) \bu (t), \\[1ex]
	\byr (t;\bp) 		 = \bCr ( \bp ) \bxr (t;\bp), 		
	\end{array} \right.
	\end{align}
where $\bEr(\bp),\bAr(\bp), \bNr_j(\bp) \in \R^{r\times r}$, for $j=1,\dots m$, $\bBr(\bp) \in \R^{r\times m}$, and $\bCr(\bp) \in \R^{\ell\times r}$ such that the reduced output $\byr (t;\bp)$ provides a good approximation to the original output $\by(t;\bp)$ for a variety of inputs $\bu(t)$ and a range of parameters $\bp$.

\emph{Non-parametric} bilinear systems where the state-space matrices $\bE$, $\bA$,
$\bN_j$ for $j=1,\ldots,m$, $\bB$ and $\bC$ are constant, have been 
    studied thoroughly, and  input-independent/optimal model reduction 
techniques from the linear case ($\bN_j = 0$ for $j=1,\ldots,m$) have been successfully generalized to non-parametric bilinear systems. 
For example, \cite{phillips2003projection,bai2006projection,breiten2010krylov,benner2011generalised,ahmad2017krylov} have extended model reduction via rational interpolation \cite{antoulas2010interpolatory,beattie2017model} from linear to non-parametric bilinear systems. 
The optimal model reduction of linear dynamical systems in the $\mathcal{H}_2$ norm via the iterative rational Krylov algorithm (IRKA) \cite{gugercin2008h_2} has been generalized to bilinear systems via bilinear IRKA (B-IRKA) \cite{benner2012interpolation}. 
Later, \cite{flagg2015multipoint}  showed that, as with IRKA and $\mathcal{H}_2$ model reduction in the linear case, the reduced model via B-IRKA also yields a Hermite interpolation in this case, but in the sense of  Volterra series interpolation. 
Similarly,  gramians and balanced truncation  (BT) for linear dynamical systems \cite{mullis1976synthesis,moore1981principal} have also been generalized to nonparametric bilinear systems \cite{al1993new,gray1998energy,hartmann2013balanced,benner2011lyapunov}. Moreover, \cite{antoulas2016model} has applied the Loewner framework \cite{mayo2007framework} to bilinear systems.

A plethora of work exists on model reduction of parametrized linear dynamical systems, i.e., $\bN_j = 0$ for $j=1,\ldots,m$ in \eqref{binon}; see, for example, \cite{baur2011interpolatory,benner2015survey,gunupudi2003ppt,Daniel2004,BenF14,Panzer_etal2010,AmsallemFarhat2011,Degroote2010,BuiThanh2008} and the references therein. In this paper, we are interested in input-independent (transfer function-based) model reduction of parametric bilinear systems where only the state-space matrices enter into the model reduction process and there is no need to choose a specific input $\bu(t)$ nor to simulate the full model \eqref{binon}. More specifically, we are focused on parametric model reduction that uses the concept of (parametric) rational interpolation. These methods, also referred to as interpolatory  parametric model reduction, have been successfully applied to parametric \emph{linear} dynamical systems; see, e.g., \cite{baur2011interpolatory,Daniel2004,BenF14}. However, unlike the extensions of interpolation theory and IRKA to \emph{non-parametric} bilinear systems,  interpolatory methods have not yet been generalized to parametric bilinear systems. In this paper, we close this gap and provide a natural extension of interpolatory projections to parametric bilinear dynamical systems.  
Our framework yields a reduced parametric bilinear model whose subsystem transfer functions will (tangentially) interpolate the original subsystem transfer functions together with the parameter sensitivities and Hessians at the sampled frequencies and parameter values along chosen directions.  Note that we are {\em not} focusing on the problem of selecting parameter samples, but rather on ensuring tangential interpolation of the full and reduced models at the chosen points and directions.  One can then couple this interpolatory model reduction algorithm to a desired sampling strategy. 

The remainder of the paper is organized as follows: Section \ref{problem} introduces the problem description and presents the main theoretical results. Section \ref{examples} illustrates the theory using two numerical  examples. This is followed by the conclusions and future directions in Section \ref{sec:conc}.


\section{Problem Description}
	\label{problem}

In this section, we introduce the ingredients of the model reduction problem for parametric bilinear systems such as projection, subsystem transfer function, and tangential interpolation.  We then present the main results of the paper.

\subsection{Projection-based model reduction of parametric bilinear systems via global basis}  \label{sec:projintro}

We construct the reduced parametric bilinear system \eqref{brom} via projection. 
We follow the \emph{global basis approach} (as opposed to using a local basis and performing extrapolation \cite{hay2009local,zimmermann2016local} or interpolation \cite{AmsallemFarhat2011,Degroote2010,zimmermann2014locally}). Thus we construct two constant global model reduction bases, namely $ \bV \in \C^{n \times r}$ and $ \bW \in \C^{n \times r}$, that capture the parametric dependence of the underlying system using the information from various sampling points.  We refer the reader to \cite{benner2015survey} for detailed explanations regarding global and local bases, and different sampling options. 
The subspaces $\bV$ and $\bW$ are computed to enforce specific interpolation conditions as discussed in Section \ref{sec:conditions}.

Once the model reduction bases $\bV$ and $\bW$ are constructed, the reduced model quantities in \eqref{brom} are obtained via Petrov-Galerkin projection:
\begin{equation} \label{eq:romss}
\begin{array}{llll}
\bEr(\bp) =  \bW^\top \bE (\bp) \bV,  & \bAr =  \bW^\top \bA (\bp) \bV, &  \bBr(\bp) =  \bW^\top \bB (\bp), \\
\bCr(\bp) = \bC (\bp) \bV , ~~~~~~\mbox{and}~~~   & \bNr_j =  \bW^\top \bN_j (\bp) \bV ~~~~\mbox{for}~j=1,\ldots,m.
\end{array}
\end{equation}
Now consider reevaluating reduced model quantities in \eqref{eq:romss} for a new parameter value $\bhp \in \R^\nu$. Consider the case of $\bEr(\bhp)$. This will require re-evaluating the projection $\bEr(\bhp) =  \bW^\top \bE (\bhp) \bV$ where the operations depend on the original system dimension $n$. In practice, many problems exhibit an affine parametric structure, which makes the projection step numerically efficient. For simplicity, continue to consider the matrix $\bE(\bp)$  only. Assume that   $\bE(\bp)$   has the following affine parametric form 
\begin{equation} \label{eq:affine}
\bE(\bp)=\bE_0+ \sum_{i=1}^N f_i(\bp)\bE_i,
\end{equation}
where $f_i$ are  scalar (nonlinear) functions reflecting the parametric dependency, and $\bE_i \in \R^{n \times n}$ for $i=0,\ldots,N$ are constant matrices. Then, the reduced matrix 
$\bE_r(\bp)$ in \eqref{eq:romss}  
is given by
\begin{eqnarray} \label{eq:affine_rom}
\bE_r(\bp)= \bW^\top\bE_0 \bV+ \sum_{i=1}^N f_i(\bp) \bW^\top\bE_i \bV, 
\end{eqnarray}
where $\bW^\top\bE_i \bV$, for $i=0,\ldots,N$ have to be computed once in an offline phase then can be recombined for efficient computation of $\bE_r(\bhp)$ in any online phase. 
The same discussion applies to other matrices in \eqref{eq:romss} as well.  
When $\bE(\bp)$ does not admit such an affine parametrization as in \eqref{eq:affine}, one usually performs an affine approximation of $\bE(\bp)$ first, usually via a matrix version of (Discrete) Empirical Interpolation Method \cite{Grepl07,Chaturantabut2010}; see \cite{benner2015survey} for details. 
We will revisit this issue in the second numerical example in Section \ref{ex:heat}.

\subsection{Interpolatory projections for parametric linear systems} \label{sec:linear}
A powerful framework in the case of linear dynamical systems
	\begin{align} 
	\label{eq:lfom}
	\bE ( \bp ) \dot{ \bx } (t;\bp) = \ds \bA (\bp) \bx (t;\bp) + \bB (\bp) \bu (t), 
	~~~~\by (t;\bp) 		 = \bC (\bp) \bx (t;\bp), 		
	\end{align}
is to transform the problem into the frequency domain via Laplace transform. To do so, let ${\bY}(s;\bp)$ and ${\bU}(s)$ denote the Laplace transforms of $\by(t;\bp)$ and $\bu(t)$, respectively. Then, applying the Laplace transform to \eqref{eq:lfom} leads to
$$
\bY(s;\bp) = \bH(s;\bp) \bU(s),~~~\mbox{where}~~~\bH(s;\bp) = \bC(\bp)\left(s\,\bE(\bp)\, -\,\bA(\bp)\right)^{-1}\bB(\bp)
$$
is the transfer function of \eqref{eq:lfom}. Then, the goal is to construct a reduced parametric linear model  
	\begin{align} 
	\label{eq:lrom}
	\bEr ( \bp ) \dot{ \bxr } (t;\bp) = \ds \bAr (\bp) \bxr (t;\bp) + \bBr (\bp) \bu (t), 
	~~~~\byr(t;\bp) 		 = \bCr (\bp) \bxr (t;\bp), 		
	\end{align}
	whose reduced parametric transfer function $\bHr(s;\bp) = \bCr(\bp)\left(s\,\bEr(\bp)\, -\,\bAr(\bp)\right)^{-1}\bBr(\bp)$ approximates  $\bH(s;\bp)$ well, which would in turn imply
	$\byr(t;\bp) \approx \by(t;\bp)$ since $\bY(s;\bp) - \bYr(s;\bp) = (\bH(s;\bp) -\bHr(s;\bp))\bU(s)$. One way to enforce  $\bHr(s;\bp)\approx  \bH(s;\bp)$ is via rational interpolation: Given the frequency interpolation points  $\{\s_1,\ldots,\s_{q_s}\} \subset \C $, the right tangential directions 	$\{\bb_1,\ldots,\bb_{q_s}\} \subset \C^m $, the left tangential directions $\{\bc_1,\ldots,\bc_{q_s}\} \subset \C^\ell$, and the parameter interpolation samples $\{\bhp_1,\ldots,\bhp_{q_p}\} \subset \R^\nu$, find a reduced model \eqref{eq:lrom} such that $\bHr(s;\bp)$ is a Hermite tangential interpolant to $\bH(s;\bp)$ at the selected samples, i.e.,
	$$
	\begin{array}{cc}
	\bH(\s_i;\bhp_j) \bb_i =\bHr(\s_i;\bhp_j) \bb_i,  & \bc_i^\top\bH(\s_i;\bhp_j)  =\bc_i^\top\bHr(\s_i;\bhp_j), \\
	\frac{\partial}{\partial s} \left( \bc_i^\top \bH(\s_i;\bhp_j) \bb_i \right) = 
	\frac{\partial}{\partial s} \left( \bc_i^\top \bHr(\s_i;\bhp_j) \bb_i \right), & ~~~\mbox{and}~~~
	\nabla_\bp \left( \bc_i^\top \bH ( \s_{i}; \bhp_j ) \bb_{i} \right)  = \nabla_\bp \left( \bc_{i}^\top \bHr_1 ( \s_{i}; \bhp_j ) \bb_i \right) ,
	\end{array}
	$$
	 for $i =1,\ldots,q_s$ and $j= 1,\ldots,q_p$. In other words, the reduced model tangentially matches transfer function values in addition to its frequency and parametric derivatives at the sampled points. One can impose higher order interpolation conditions at the frequency and parameter samples as well, such as the parameter Hessian. 
We omit it for brevity here.   
The following result from \cite{baur2011interpolatory} shows how to construct model reduction bases $\bV$ and $\bW$ that satisfy the desired interpolation conditions. 
	 
	 \begin{theorem} \label{thm:linear} 
	 Given $\bH(s;\bp) =\bC(\bp)\left(s\,\bE(\bp)\, -\,\bA(\bp)\right)^{-1}\bB(\bp) $, let  $\bHr(s;\bp) = \bCr(\bp)\left(s\,\bEr(\bp)\, -\,\bAr(\bp)\right)^{-1}\bBr(\bp)$ be obtained via Petrov-Galerkin projection using the bases $\bV$ and $\bW$.  
Let $ \s \in \C $, $ \bhp \in \R^\nu $, $ \bb \in \C^m \setminus \{ \bf0 \} $, and $ \bc \in \C^\ell \setminus \{ \bf0 \} $. Define 
\begin{equation} \label{eq:cA}
\cA ( s; \bp )  = s \bE ( \bp ) - \bA ( \bp ) 	 \quad \mbox{and}\quad
	\cAr ( s; \bp )  = s \bEr ( \bp ) - \bAr ( \bp ).
\end{equation}		
	\begin{enumerate}[(a)]
	\item If $ \cA (\s; \bhp)^{-1} \bB (\bhp) \bb \in \textup{Ran} (\bV) $, then $$ \bH(\s; \bhp) \bb = \bHr (\s; \bhp) \bb;$$
	\item If $ \cA (\s; \bhp)^{-\top} \bC (\bhp)^\top \bc \in \textup{Ran} (\bW) $, then $$ \bc^\top \bH (\s; \bhp) = \bc^\top \bHr (\s; \bhp);$$
	\item If  both (a) and (b) hold simultaneously, then 
		$$ \frac{\partial}{\partial s} \left( \bc^\top \bH (\s; \bhp) \bb \right)
		 = \frac{\partial}{\partial s} \left( \bc^\top \bHr (\s; \bhp) \bb \right)
		 \quad\mbox{and}\quad
		\nabla_\bp \left( \bc^\top \bH (\s; \bhp) \bb \right) 
		= \nabla_\bp \left( \bc^\top \bHr (\s; \bhp) \bb \right), $$
	\end{enumerate}
	provided $ \cA (\s; \bhp) $ and $ \cAr (\s; \bhp) $ are invertible.
\end{theorem}
Theorem \ref{thm:linear} shows  how to construct $\bV$ and $\bW$ to fulfill the required interpolation conditions. All one has to do is to compute the vectors, e.g.,  the vector $ \cA (\s; \bhp)^{-1} \bB (\bhp) \bb$, for the desired frequency interpolation points $\s$ and parameter interpolation point $\bhp$, and use these vectors as columns of $\bV$. We refer the reader to the original source \cite{baur2011interpolatory} for more details. The goal of this paper is to extend this result to parametric bilinear systems.

\subsection{Interpolatory parametric bilinear model reduction problem} \label{sec:intproblem}
Re-consider the full-order parametric bilinear system in \eqref{binon}:
	\begin{align} 
	\left\{ \begin{array}{l}
	\tag{\ref{binon}}
	\bE ( \bp ) \dot{ \bx } (t;\bp) 	 = \ds \bA (\bp) \bx (t;\bp) + \sum_{j=1}^m \bN_j (\bp) \bx (t) u_j (t) + \bB (\bp) \bu (t), \\[1ex]
	\by (t;\bp) 		 = \bC (\bp) \bx (t;\bp). 		
	\end{array} \right.
	\end{align}
Even though this system is nonlinear, due to the terms involving $\bN_j(\bp)$, the concept of transfer function can still be applied  via Volterra series representation \cite{rugh1981nonlinear}.  Given the bilinear system 
\eqref{binon}, we first introduce some notation to make the presentation of the Volterra series representation more compact:
	\begin{align} \label{eq:not1}
	\bN (\bp) & = [ \bN_1 (\bp) \ \bN_2 (\bp) \ \cdots \ \bN_m (\bp) ] ,				&
	\overline{ \bN } (\bp)  & = \left[ \begin{array}{c} \bN_1 (\bp) \\ \bN_2 (\bp) \\ \vdots \\ \bN_m (\bp)  \end{array} \right],   ~~\mbox{and}~~&
	\bI_m^{\otimes^k}  & = \underbrace{ \bI_m \otimes \cdots \otimes \bI_m }_{k \text{ times }},	\end{align}
where
$\otimes$ denotes the Kronecker product.

The output $\by(t;\bp)$ of \eqref{binon} can be represented as a Volterra series
 \begin{equation} 
\mathbf{y}(t;\bp)=\sum_{k=1}^{\infty}\int_0^{t_1}\int_0^{t_2}\cdots \int_0^{t_{k}}\mathbf{h}_k(t_1,t_2,\dots,t_k;\bp)\left(\mathbf{u}(t-\sum_{i=1}^k t_i)\otimes \cdots \otimes \mathbf{u}(t-t_k)\right) \mbox{d} t_k \cdots \mbox{d} t_1,
\end{equation}
where $\mathbf{h}_k(t_1,t_2,\dots,t_k;\bp)$'s are the regular Volterra kernels, also called subsystem kernels. 
Then, taking
the multivariable Laplace transform of the degree $k$ regular kernel $\mathbf{h}_k$
leads to the $k^{\rm th}$ subsystem transfer function: 
	\begin{align}
	\bH_k ( s_1, \dots, s_k; \bp ) = 
		 ~ \bC (\bp) \cA (s_k; \bp)^{-1} \times\bN (\bp) 
		[ \bI_m \otimes \cA (s_{k-1}; \bp)^{-1} \bN (\bp) ] 
		\cdots [ \bI_m^{\otimes^{k-2}} \otimes &\cA (s_2; \bp)^{-1} \bN (\bp) ] \label{eq:Hkfom}	\\
		& \times [ \bI_m^{\otimes^{k-1}} \otimes \cA (s_1; \bp)^{-1} \bB (\bp) ],	\nonumber
		\end{align}
		where $\cA (s; \bp)$ is as defined in \eqref{eq:cA}, and $\bN(\bp)$ and $\bI_m^{\otimes^{k}}$
		are as defined in \eqref{eq:not1}. 
For details of this analysis, we refer the reader to \cite{siu1991convergence,rugh1981nonlinear}.
The Volterra series representation of bilinear systems has been successfully used for interpolation-based input-independent, optimal model reduction of non-parametric bilinear systems; see, e.g., \cite{benner2012interpolation,flagg2015multipoint}.

Similarly, for the reduced bilinear system \eqref{brom}, define 
\begin{align}	\label{eq:not2}
	\bNr (\bp) & = [ \bNr_1 (\bp) \ \bNr_2 (\bp) \ \cdots \ \bNr_m (\bp) ].		\end{align}
Then, the $k^{\rm th}$ subsystem transfer function of the reduced model \eqref{brom} is given by
	\begin{align}  \label{eq:Hrfom}
	\bHr_k ( s_1, \dots, s_k; \bp ) = 
		 ~ \bCr (\bp) \cAr (s_k; \bp)^{-1} {\color{black}\times}\bNr (\bp) 
		[ \bI_m \otimes \cAr (s_{k-1}; \bp)^{-1} \bNr (\bp) ] 
		\cdots [ \bI_m^{\otimes^{k-2}} \otimes & \cAr (s_2; \bp)^{-1} \bNr (\bp) ] 	\\
		& \times [ \bI_m^{\otimes^{k-1}} \otimes \cAr (s_1; \bp)^{-1} \bBr (\bp) ]. \nonumber	
	\end{align}

This allows us to formulate the parametric interpolatory model reduction problem in our setting:
	Given interpolation frequencies $ \{ \s_1, \dots, \s_q \} \subset \C $, 
	nontrivial right direction $ \bb \in \C^m $, 
	nontrivial left direction $ \bc \in \C^\ell $, and 
	interpolation parameter sample $ \bhp \in \R^\nu $,
find  $ \bV, \bW $  such that the reduced model \eqref{brom} constructed via projection as in \eqref{eq:romss}
satisfies the following interpolation conditions
for any $ k \in \{ 1, \dots, q \}$:
	\begin{align}
	\bH_k ( \s_1, \dots, \s_k; \bhp ) ( \bI_m^{\otimes^{k-1}} \otimes \bb )	
		& = \bHr_k ( \s_1, \dots, \s_k; \bhp ) ( \bI_m^{\otimes^{k-1}} \otimes \bb ), 
		\label{eq:rint} \\
	\bc^\top \bH_k ( \s_1, \dots, \s_k; \bhp )	
		& = \bc^\top \bHr_k ( \s_1, \dots, \s_k; \bhp ), \label{eq:lint}  \\
	\frac{\partial}{\partial s_i} \left( \bc^\top \bH_k ( \s_1, \dots, \s_k; \bhp ) ( \bI_m^{\otimes^{k-1}} \otimes \bb ) \right)
		& = \frac{\partial}{\partial s_i} \left( \bc^\top \bHr_k ( \s_1, \dots, \s_k; \bhp ) ( \bI_m^{\otimes^{k-1}} \otimes \bb ) \right) ,
		& & i \in \{ 1, \dots, k \}  \label{eq:hermite} \\
	\cbJ_\bp \left( \bc^\top \bH_k ( \s_1, \dots, \s_k; \bhp ) ( \bI_m^{\otimes^{k-1}} \otimes \bb )	 \right)
		& = \cbJ_\bp \left(  \bc^\top \bHr_k ( \s_1, \dots, \s_k; \bhp ) ( \bI_m^{\otimes^{k-1}} \otimes \bb ) \right),
		\label{eq:gradient} \\
	\cbH_\bp \left( \bc^\top \bH_k ( \s_1, \dots, \s_k; \bhp ) ( \bI_m^{\otimes^{k-1}} \otimes \bb )\right)
		& = \cbH_\bp \left(\bc^\top \bHr_k ( \s_1, \dots, \s_k; \bhp ) ( \bI_m^{\otimes^{k-1}} \otimes \bb ) \right), 
		\label{eq:hessian} 
	\end{align}
where $\cbJ_{\bp}(\cdot)$ denotes the matrix of sensitivities (Jacobian) and $\cbH_{\bp}(\cdot)$ denotes the Hessian (tensor) with respect to $\bp$. 
In other words, we would like to construct a reduced parametric bilinear system whose leading subsystems interpolate (both in frequency and parameter space) the corresponding leading subsystems of 
	the full order parametric model. Note that we are not only enforcing Lagrange interpolation. We require the reduced model to match the parameter sensitivities and Hessians as well, which is important, especially, in the setting of reduced models in optimization. 
Moreover, these conditions can then be generalized for different reordering of the frequencies, multiple tangential directions, and several parameter values.

\subsection{Subspace conditions for parametric bilinear interpolation} \label{sec:conditions}
In this section, we establish the subspace conditions to enforce the desired interpolation conditions 
\eqref{eq:rint}--\eqref{eq:hessian} for parametric bilinear systems. 
Note that even for the parametric bilinear system \eqref{binon} we consider here, some of these interpolation conditions, e.g, \eqref{eq:rint}, do not involve parameter gradient and/or parameter Hessian interpolation; and  thus can be interpreted as regular tangential bilinear subsystem interpolation for a fixed parameter $\bp = \bhp$ as considered in \cite{benner2011generalised}. However, even though our subspace conditions for \eqref{eq:rint}-\eqref{eq:lint} will look similar to those in
\cite{benner2011generalised}, we include the corresponding theorem  (Theorem \ref{thm:pbmor1sided} below) and its complete proof for the following reasons. 
Although \cite{benner2011generalised} considers tangential interpolation for non-parametric bilinear systems, the tangential interpolation conditions appear differently. In our formulation, tangential directions appear in Kronecker product form due to the structure of $\bH_k (s_1,s_2,\ldots,s_k;\bp)$ as defined in 
\eqref{eq:Hkfom} and illustrates that $\bH_k (s_1,s_2,\ldots,s_k;\bp)$ can be considered to have $m^k$ inputs.  
Our conditions result in regular tangential interpolation in the subblocks of $\bH_k (s_1,s_2,\ldots,s_k;\bp)$; details will be given below. Moreover, we provide different proofs for \eqref{eq:rint} and \eqref{eq:lint}, which we later use in the proof of Theorem \ref{thm:pbmor2sided}. Finally, we include the $\bE(\bp)$ term in the full model. Clearly, the subspaces conditions \eqref{eq:gradient} for matching the parameter gradient and \eqref{eq:hessian} for matching the parameter Hessian  are new and will be fully discussed.

\begin{theorem} \label{thm:pbmor1sided}
Let $ q $ be the number of subsystems we wish to interpolate.
Consider $ \{ \s_1, \dots, \s_q \} \subset \C $ and $ \bhp \in \R^\nu $ such that $ \cA ( \s_i; \bhp ) $ is invertible for all $ i \in \{ 1, \dots, q \} $.
Consider also the nontrivial vectors $ \bb \in \C^m $ and $ \bc \in \C^\ell $.
Define
	\begin{align}
	\bV_1 		& = \cA ( \s_1; \bhp )^{-1} \bB (\bhp) \bb , &
	\bV_k 		& = \cA ( \s_k; \bhp )^{-1} \bN (\bhp) ( \bI_m \otimes \bV_{k-1} ) , 				&  \mbox{for}~~k = 2, \dots, q , \label{eq:defineV} \\
	\bW_1 	& = \left( \cA ( \s_q; \bhp ) \right)^{-\top} \bC (\bhp)^\top\bc,&
	\bW_k 	& = 
	 \left( \cA ( \s_{q+1-k}; \bhp ) \right)^{-\top} \overline{ \bN}(\bhp)^\top
	( \bI_m \otimes \bW_{k-1}), 	& \mbox{for}~~k = 2, \dots, q .  \label{eq:defineW} 
	\end{align}	
If 
	\begin{equation}  \label{eq:Vcond}
	\bigcup_{k=1}^q \bV_k \subseteq \textup{Ran} ( \bV ) ,
	\end{equation}
then, for $ k = 1, \dots, q$,
	\begin{equation}  \label{eq:Vint}
	\bH_k ( \s_1, \dots, \s_k; \bhp ) ( \bI_m^{\otimes^{k-1}} \otimes \bb ) = \bHr_k ( \s_1, \dots, \s_k; \bhp ) ( \bI_m^{\otimes^{k-1}} \otimes \bb ) .
	\end{equation}
If
	\begin{equation} \label{eq:Wcond}
	\bigcup_{k=1}^q \bW_k \subseteq \textup{Ran} ( \bW ) ,
	\end{equation}
then, for $ k =1, \dots, q$,
	\begin{equation} \label{eq:Wint}
	\bc^\top \bH_k ( \s_{q+1-k}, \dots, \s_q; \bhp ) = \bc^\top \bHr_k ( \s_{q+1-k}, \dots, \s_q; \bhp ) .
	\end{equation}
\end{theorem}

\begin{proof}
Define
	\begin{align} \label{eq:pqfg}
	\cP ( s; \bp ) & = \bV \cAr ( s; \bp )^{-1} \bW^\top \cA ( s; \bp ) , \nonumber\\
	\cQ ( s; \bp ) & = \cA ( s; \bp ) \bV \cAr ( s; \bp )^{-1} \bW^\top ,\nonumber\\
	\bff_k ( s_1, \dots, s_k; \bp ) 
		& = \cA (s_k; \bp)^{-1} \bN (\bp) ( \bI_m \otimes \cA (s_{k-1}; \bp)^{-1} \bN (\bp) ) \cdots 
		( \bI_m^{\otimes^{k-2} } \otimes \cA (s_2; \bp)^{-1} \bN (\bp) ) \\
		& \quad\quad~ \times( \bI_m^{\otimes^{k-1} } \otimes \cA (s_1; \bp)^{-1} \bB (\bp) \bb ),~\mbox{and}\nonumber\\
	\bg_k^\top ( s_1, \dots, s_k; \bp ) 
		& = \bc^\top \bC (\bp) \cA (s_k; \bp)^{-1} \bN (\bp) ( \bI_m \otimes \cA (s_{k-1}; \bp)^{-1} \bN (\bp) ) \cdots 
		( \bI_m^{\otimes^{k-2} } \otimes \cA (s_2; \bp)^{-1} \bN (\bp) ) \nonumber\\
		& \quad\quad\times ( \bI_m^{\otimes^{k-1} } \otimes \cA (s_1; \bp)^{-1} ) .\nonumber
	\end{align}
Note that $ \cP (s;\bp) $ is a skew projector onto $ \textup{Ran} (\bV) $ while $ \cQ (s;\bp) $ is a skew projector along $ \textup{Ker} (\bW^\top) $. 
Also note that 
	$ \bH_k (s_1,\ldots,s_k;\bp) ( \bI_m^{\otimes^{k-1}} \otimes \bb ) = \bC (\bp)\bff_k ( s_1, \dots, s_k; \bp ) $
and
	$ \bc^\top \bH_k (s_1,\ldots,s_k;\bp) = \bg_k^\top ( s_1, \dots, s_k; \bp ) ( \bI_m^{\otimes^{k-1}} \otimes \bB (\bp) ) $.

First we prove \eqref{eq:Vint}. Suppose \eqref{eq:Vcond} holds. 
We know that  \eqref{eq:Vint} is true for $ k = 1 $ by Theorem \ref{thm:linear}.
Assume that the result is true for $ k -1 $; recall $k<q$.
Then using the definitions of $\cP ( s; \bp )$ and $\bff_k ( s_1, \dots, s_k; \bp )$ from \eqref{eq:pqfg}, we obtain	\begin{equation} \label{intHk}
	\bH_k ( \s_1, \dots, \s_k; \bhp ) ( \bI_m^{\otimes^{k-1}} \otimes \bb ) - \bHr_k ( \s_1, \dots, \s_k; \bhp ) ( \bI_m^{\otimes^{k-1}} \otimes \bb )
	= \bC (\bhp) ( \bI_n - \cP ( \s_k; \bhp ) ) \bff_k ( \s_1, \dots, \s_k; \bhp) ,
	\end{equation}
where we factor out the term $ \bff_k ( \s_1, \dots, \s_k; \bhp) $ using the right interpolation of $ \bH_{k-1} (\s_1, \dots, \s_{k-1}; \bhp) $ due to the induction assumption.
Then, what is left to show is that \eqref{intHk} is zero:
By  the construction of $ \bV $ in \eqref{eq:Vcond}, we obtain $ \bff_k ( \s_1, \dots, \s_k; \bhp) \in \textup{Ran} ( \bV ) $.
Hence $ \bff_k ( \s_1, \dots, \s_k; \bhp) \in \textup{Ran} ( \cP ( \s_k; \bhp ) )$, 
 which implies
	\begin{equation}
	\label{intPf}
	( \bI_n - \cP ( \s_k; \bhp ) ) \bff_k ( \s_1, \dots, \s_k; \bhp) = 0,
	\end{equation}
since $ \cP (s;\bp) $ is a skew projector onto $ \textup{Ran} (\bV) $. The proof of 
\eqref{eq:Wint} follows similarly. Suppose \eqref{eq:Wcond} holds.
Once again, the result is true for $ k = 1 $.  
Assume that it holds for $ k-1$.
Similar to \eqref{intHk}, we obtain
	\begin{equation} \label{LintHk}
	\bc^\top \bH_k ( \s_{q+1-k}, \dots, \s_q; \bhp ) - \bc^\top \bHr_k ( \s_{q+1-k}, \dots, \s_q; \bhp ) 
	= \bg_k^\top ( \s_{q+1-k}, \dots, \s_q; \bhp) [ \bI_m^{\otimes^{k-1}} \otimes ( \bI_n - \cQ ( \s_q; \bhp ) ) \bB (\bhp) ].
	\end{equation}
Again we show that this expression is zero.
Note that by the definition of $ \bg_k$ and the construction of  $ \bW $ in \eqref{eq:Wcond},
 we have $ \bg_k ( \s_{q+1-k}, \dots, \s_q; \bhp) \perp \textup{Ker} ( \bI_m^{\otimes^{k-1}} \otimes \cQ ( \s_q; \bhp ) ) $; thus
	\begin{equation} 
	\label{intQg}
	\bg_k^\top ( \s_{q+1-k}, \dots, \s_q; \bhp) [ \bI_m^{\otimes^{k-1}} \otimes ( \bI_n - \cQ ( \s_q; \bhp ) ) ] = 0 .
	\end{equation}
	\hfill $\square$
\end{proof}
Theorem \ref{thm:pbmor1sided} provides tangential interpolation of $\bH_k(s_1,\ldots,s_k;\bp)$ in a specific order of the frequencies, namely in the order $\{\s_1,\ldots,\s_k\}$. However one might also consider enforcing interpolation  at the frequency samples $\{\s_1,\ldots,\s_k\}$ in any order, including repetitions, as is considered in \cite{benner2011generalised}. Indeed, as we will show in Theorem \ref{thm:pbmor2sided}, interpolation of the transfer function sensitivities will require this. The result is a direct extension of Theorem \ref{thm:pbmor1sided}; thus we skip the details. It simply requires the subspaces to contain all possible combinations:
\begin{corollary} \label{remek:order}
Let $ q $ be the number of subsystems we wish to interpolate.
Consider $ \{ \s_1, \dots, \s_q \} \subset \C $ and $ \bhp \in \R^\nu $ such that $ \cA ( \s_i; \bhp ) $ is invertible for all $ i \in \{ 1, \dots, q \} $.
Consider also the nontrivial vectors $ \bb \in \C^m $ and $ \bc \in \C^\ell $. Define
\begin{align} 
	\begin{aligned} 
	\label{hyp}
	\bV_1 	& = [ \cA ( \s_1; \bhp )^{-1} \bB (\bhp) \bb,  \ \ \cdots, \ \ \cA ( \s_q; \bhp )^{-1} \bB (\bhp) \bb ], \\
	\bV_k 	& = [ \cA ( \s_1; \bhp )^{-1} \bN (\bhp) ( \bI_m \otimes \bV_{k-1} ), \ \ \cdots, \ \ \cA ( \s_q; \bhp )^{-1} \bN (\bhp) ( \bI_m \otimes \bV_{k-1} ) ] , 				& k = 2, \dots, q , \\
	\bW_1 	& = [ \left( \cA ( \s_1; \bhp ) \right)^{-\top} \bC (\bhp)^\top\bc, \ \ \cdots, \ \ 
				\left( \cA ( \s_q; \bhp ) \right)^{-\top} \bC (\bhp)^\top\bc ] \\
	\bW_k 	& = [ \left( \cA ( \s_1; \bhp ) \right)^{-\top} \overline{ \bN}(\bhp)^\top ( \bI_m \otimes \bW_{k-1}), 
				\ \ \cdots, \ \ 
				\left( \cA ( \s_q; \bhp ) \right)^{-\top} \overline{ \bN}(\bhp)^\top ( \bI_m \otimes \bW_{k-1}) ] , 	
				& k = 2, \dots, q . 
	\end{aligned}
	\end{align}
If 
	\begin{equation}  \label{eq:Vcondall}
	\bigcup_{k=1}^q \bV_k \subseteq \textup{Ran} ( \bV ) ,
	\end{equation}
then, for $ k = 1, \dots, q $, and for any $ i_1, \dots, i_k \in \{ 1, \dots, q \} $,
	\[  
	\bH_k ( \s_{i_1}, \dots, \s_{i_k}; \bhp ) ( \bI_m^{\otimes^{k-1}} \otimes \bb ) 
	= \bHr_k ( \s_{i_1}, \dots, \s_{i_k}; \bhp ) ( \bI_m^{\otimes^{k-1}} \otimes \bb ).
	\]
If
	\begin{equation} \label{eq:Wcondall}
	\bigcup_{k=1}^q \bW_k \subseteq \textup{Ran} ( \bW ) ,
	\end{equation}
then, for $ k = 1, \dots, q $, and for any $ i_1, \dots, i_k \in \{ 1, \dots, q \} $,
	\[  
	\bc^T \bH_k ( \s_{i_1}, \dots, \s_{i_k}; \bhp ) 
	= \bc^T \bHr_k ( \s_{i_1}, \dots, \s_{i_k}; \bhp ).
	\]
	\end{corollary}
So far, we proved the interpolation conditions using either only $\bV$ or only $\bW$; i.e., we assumed interpolation information only in one of the subspaces, considering one-sided projection. The next theorem shows that when both subspaces are considered one automatically matches the sensitivities (derivatives) with respect to the frequencies and  parameter, indeed without computing the sensitivities to be matched. 
\begin{theorem} \label{thm:pbmor2sided}
Assume the hypotheses of Corollary \ref{remek:order}.
 Let  $\bV_k$ and $\bW_k$ be constructed as in  \eqref{hyp} for $k=1,2,\ldots,q$.
If both
	\[
	\bigcup_{k=1}^q \bV_k \subseteq \textup{Ran} ( \bV ) 
	\qquad and \qquad
	\bigcup_{k=1}^q \bW_k \subseteq \textup{Ran} ( \bW ) ,
	\]
then for $ k = 1, \dots, q $ and  for $ i = 1, \dots, k $: 
	\[ 
	\frac{\partial}{\partial s_i} \left( \bc^\top \bH_k ( \s_1, \dots, \s_k; \bhp ) ( \bI_m^{\otimes^{k-1}} \otimes \bb ) \right)
	= \frac{\partial}{\partial s_i} \left( \bc^\top \bHr_k ( \s_1, \dots, \s_k; \bhp ) ( \bI_m^{\otimes{k-1}} \otimes \bb ) \right) , 
	\]
and 
	\[ 
	\cbJ_{\bp} \left(\bc^\top \bH_k ( \s_1, \dots, \s_k; \bhp ) ( \bI_m^{\otimes^{k-1}} \otimes \bb ) \right)
	= \cbJ_{\bp} \left(\bc^\top \bHr_k ( \s_1, \dots, \s_k; \bhp ) ( \bI_m^{\otimes^{k-1}} \otimes \bb ) \right).
	\]
	where $\cbJ_{\bp}(\cdot)$ denotes the matrix of sensitivities (Jacobian) with respect to $\bp$.
\end{theorem}

\begin{proof}
Recall the definitions $ \cP (s;\bp) $, $ \cQ (s;\bp) $, $ \bff_k (s_1,\dots,s_k;\bp) $, and $ \bg_k^\top (s_1,\dots,s_k;\bp) $ from \eqref{eq:pqfg}. Similarly,
define
	\begin{align*}
	\bffr_k ( s_1, \dots, s_k; \bp ) 
		& = \cAr (s_k; \bp)^{-1} \bNr (\bp) ( \bI_m \otimes \cAr (s_{k-1}; \bp)^{-1} \bNr (\bp) ) \cdots 
		( \bI_m^{\otimes^{k-2} } \otimes \cAr (s_2; \bp)^{-1} \bNr (\bp) ) \\
		& \quad\quad~ \times( \bI_m^{\otimes^{k-1} } \otimes \cAr (s_1; \bp)^{-1} \bBr (\bp) \bb ) ,~\mbox{and}\\
	\bgr_k^\top ( s_1, \dots, s_k; \bp ) 
		& = \bc^\top \bCr (\bp) \cAr (s_k; \bp)^{-1} \bNr (\bp) ( \bI_m \otimes \cAr (s_{k-1}; \bp)^{-1} \bNr (\bp) ) \cdots 
		( \bI_m^{\otimes^{k-2} } \otimes \cAr (s_2; \bp)^{-1} \bNr (\bp) ) \\
		& \quad\quad~\times ( \bI_m^{\otimes^{k-1} } \otimes \cAr (s_1; \bp)^{-1} ) .
	\end{align*}
Let $ k \in \{ 1, \dots, q \} $ and
 $ i \in \{ 1, \dots, k \} $.
Under the assumptions of the theorem, we know \eqref{intPf} and \eqref{intQg} are satisfied for any choice of frequencies due to Corollary \ref{remek:order}; in particular,
	\begin{align*}
	( \bI_n - \cP ( \s_\kappa; \bhp ) ) \bff_\kappa ( \s_1, \dots, \s_\kappa; \bhp) & = 0 , 	~~~\mbox{and}~~~	&
	\bg_{\kappa-\iota+1}^\top ( \s_\iota, \dots, \s_\kappa; \bhp) [ \bI_m^{\otimes^{\kappa-\iota}} \otimes ( \bI_n - \cQ ( \s_\kappa; \bhp ) ) ] & = 0 ,
	\end{align*}
for any $ \kappa \in \{ 1, \dots, q \} $ and $ \iota < \kappa $.
Recalling the definitions of the full and reduced transfer functions \eqref{eq:Hkfom} and \eqref{eq:Hrfom}, together with \eqref{intHk} and \eqref{LintHk}, and our definitions in this proof, we can rewrite these two terms as
	\begin{align*}
	\bff_\kappa ( \s_1, \dots, \s_\kappa; \bhp) - \bV \bffr_\kappa ( \s_1, \dots, \s_\kappa; \bhp) & = 0 , 		&
	\bg_{\kappa-\iota+1}^\top ( \s_\iota, \dots, \s_\kappa; \bhp) - \bgr_{\kappa-\iota+1}^\top ( \s_\iota, \dots, \s_\kappa; \bhp) 
		( {\color{black}\bI_m^{\otimes^{\kappa-\iota}} \otimes }\bW^\top ) & = 0,
	\end{align*} 
or equivalently
	\begin{equation} 
	\label{VWint}
	\bff_\kappa ( \s_1, \dots, \s_\kappa; \bhp)  = \bV \bffr_\kappa ( \s_1, \dots, \s_\kappa; \bhp) ~~\mbox{and}~~
	\bg_{\kappa-\iota+1}^\top ( \s_\iota, \dots, \s_\kappa; \bhp) 
	 = \bgr_{\kappa-\iota+1}^\top ( \s_\iota, \dots, \s_\kappa; \bhp) ( \bI_m^{\otimes^{\kappa-\iota}} \otimes \bW^\top ) .
	\end{equation}
Now fix $ k \in \{ 1, \dots, q \} $ and $ i \in \{ 1, \dots, k \} $, and recall that $ \bEr (\bp) = \bW^\top \bE (\bp) \bV $.
If $ i \ne 1 $, then
	\begin{align*}
	\frac{\partial}{\partial s_i} \left( \bc^\top \bHr_k ( \s_1, \dots, \s_k; \bhp ) ( \bI_m^{\otimes^{k-1}} \otimes \bb ) \right)
	& = \bgr^\top_{k-i+1} (\s_i, \dots, \s_k;\bhp) ( \bI_m^{\otimes^{k-i}} \otimes \bW^\top ) \bE (\bhp) 
			\bV \bffr_i (\s_1,\dots,\s_i;\bhp) \\
	& = \bg^\top_{k-i+1} (\s_i, \dots, \s_k;\bhp) \bE (\bhp) \bff_i (\s_1,\dots,\s_i;\bhp) 	& (\textup{by } \eqref{VWint})\\
	& = \frac{\partial}{\partial s_i} \left( \bc^\top \bH_k ( \s_1, \dots, \s_k; \bhp ) ( \bI_m^{\otimes^{k-1}} \otimes \bb ) \right) .
	\end{align*} 
If $ i = 1 $, then
	\begin{align*}
	\frac{\partial}{\partial s_i} \bc^\top \bHr_k ( \s_1, \dots, \s_k; \bhp ) ( \bI_m^{\otimes^{k-1}} \otimes \bb )
	& = \bgr^\top_k (\s_1, \dots, \s_k;\bhp) ( \bI_m^{\otimes^{k-1}} \otimes \bW^\top ) \bE (\bhp) \bV \bffr_1 (\s_1;\bhp) \\
	& = \bg^\top_k (\s_1, \dots, \s_k;\bhp) \bE (\bhp) \bff_1 (\s_1;\bhp) 	& (\textup{by } \eqref{VWint}) \\
	& = \frac{\partial}{\partial s_i} \bc^\top \bH_k ( \s_1, \dots, \s_k; \bhp ) ( \bI_m^{\otimes^{k-1}} \otimes \bb ).
	\end{align*}

Similarly, we can justify interpolation of the parameter gradients. 
Since the expression of the parameter gradient for a general subsystem transfer function $ \bH_k(s_1,\ldots,s_k;\bp)$ becomes too involved to properly present in a single page, we provide the proof for the second subsystem (the result for the first subsystem follows Theorem \ref{thm:linear}) and sketch the proof for a general subsystem.
Let $ p_j $ refer to any entry of the parameter vector $ \bp \in \R^\nu$.
Consider
\begin{equation} \label{eq:hh}
\frac{\partial}{\partial p_j} \left( \bc^\top \bH_2 (\s_1,\s_2;\bhp) ( \bI_m \otimes \bb ) -  \bc^\top \bHr_2 (\s_1,\s_2;\bhp) ( \bI_m \otimes \bb ) \right).
\end{equation}
Let $\mathbf{M}_{p_j}(\bp)$ denote the partial derivative of $\mathbf{M}(\bp)$ with respect to $p_j$. Then, 
by taking the partial derivatives in \eqref{eq:hh}, using interpolation of the first subsystem, rearranging terms and using 
	$ \bCr_{p_j} (\bp) = \bC_{p_j} (\bp) \bV $, $ \cAr_{p_j} (s;\bp) = \bW^\top \cA_{p_j} (s;\bp) \bV $, $ \bNr_{p_j} (\bp) = \bW^\top \bN_{p_j} (\bp) ( \bI_m \otimes \bV ) $, and $ \bBr_{p_j} (\bp) = \bW^\top \bB_{p_j} (\bp) $,
one can write
	\begin{align*}
	& \frac{\partial}{\partial p_j} \left( \bc^\top \bH_2 (\s_1,\s_2;\bhp) ( \bI_m \otimes \bb ) -  \bc^\top \bHr_2 (\s_1,\s_2;\bhp) ( \bI_m \otimes \bb ) \right) \\
	& \qquad = ( \bc^\top \bC_{p_j} (\bhp) - \bg_1 ( \s_2; \bhp) \cA_{p_j} ( \s_2; \bhp ) ) ( \bI - \cP ( \s_2; \bhp ) ) \bff_2 ( \s_1, \s_2; \bhp) \\
	& \qquad\qquad + \bg_2^\top ( \s_1, \s_2; \bhp ) \left( \bI_m \otimes ( \bI - \cQ ( \s_1; \bhp) ) ( \bB_{p_j} (\bhp) \bb - \cA_{p_j} ( \s_1; \bhp ) \bff_1 ) \right) ,
	\end{align*}
which can be justified by multiplying out the right-hand side and re-grouping. 
We know that in the second-line of this expression, we obtain
	 $( \bI - \cP ( \s_2; \bhp ) ) \bff_2 ( \s_1, \s_2; \bhp) =0 $ 
and in the third-line of this expression, we obtain
	$\bg_2^\top ( \s_1, \s_2; \bhp ) \left( \bI_m \otimes ( \bI - \cQ ( \s_1; \bhp) ) \right) = 0 $ 
 using $ k = 2 $ in the proof of Theorem \ref{thm:pbmor1sided}. Thus,
 we have
	 $$ 
	\frac{\partial}{\partial p_j} \left( \bc^\top \bH_2 (\s_1,\s_2;\bhp) ( \bI_m \otimes \bb )\right) =\frac{\partial}{\partial p_j} \left(   \bc^\top \bHr_2 (\s_1,\s_2;\bhp) ( \bI_m \otimes \bb ) \right).
	$$
Since $p_j$ was an arbitrary entry of $\bp$, this yields 
$\cbJ_{\bp} \left(\bc^\top \bH_2 ( \s_1, \s_2; \bhp ) ( \bI_m  \otimes \bb ) \right) 
	= \cbJ_{\bp} \left( \bc^\top \bHr_2 ( \s_1, \s_2; \bhp ) ( \bI_m \otimes \bb ) \right)$ as desired. 
Now, for the general case, consider 
	\begin{align} \label{eq:l1}
\frac{\partial}{\partial p_j}\left( \bc^\top \bHr_k ( \s_1, \dots, \s_k; \bhp ) ( \bI_m^{\otimes^{k-1}} \otimes \bb ) \right)	& = \bc^\top \bC_{p_j} (\bhp) \bV \bffr_k (\s_1, \dots, \s_k;\bhp) \\
		& \qquad + \bgr^\top_1 (\s_k;\bhp) \bW^\top \cA_{p_j} (\s_k;\bhp) \bV \bffr_k (\s_1, \dots, \s_k;\bhp)\label{eq:l2} \\
		& \qquad + \bgr^\top_1 (\s_k;\bhp) \bW^\top \bN_{p_j} (\bhp) ( \bI_m \otimes \bV ) \bffr_{k-1} (\s_1, \dots, \s_{k-1};\bhp) \label{eq:l3}\\
		& \qquad + \dots  + \bgr^\top_k (\s_1, \dots, \s_k; \bhp) ( \bI_m^{\otimes^{k-1}} \otimes \bW^\top \bB_{p_j} (\bhp) \bb ). \label{eq:l4}	\end{align}
Consider the right-hand side of \eqref{eq:l1}. Using  \eqref{VWint}, one can replace
$\bV \bffr_k (\s_1, \dots, \s_k;\bhp)$ with $\bff_k (\s_1, \dots, \s_k;\bhp)$. Similarly, in \eqref{eq:l2}, 
once again using \eqref{VWint}, one replaces $ \bgr^\top_1 (\s_k;\bhp) \bW^\top$ with $\bg^\top_1 (\s_k;\bhp)$. Continuing in this fashion, we obtain
	\begin{align*}
\frac{\partial}{\partial p_j}\left( \bc^\top \bHr_k ( \s_1, \dots, \s_k; \bhp ) ( \bI_m^{\otimes^{k-1}} \otimes \bb ) \right)
	& = \bc^\top \bC_{p_j} (\bhp) \bff_k (\s_1, \dots, \s_k;\bhp) \\
		& \qquad + \bg^\top_1 (\s_k;\bhp) \cA_{p_j} (\s_k;\bhp) \bff_k (\s_1, \dots, \s_k;\bhp) \\
		& \qquad + \bg^\top_1 (\s_k;\bhp) \bN_{p_j} (\bhp) \bff_{k-1} (\s_1, \dots, \s_{k-1};\bhp) \\
		& \qquad + \dots + \bg^\top_k (\s_1, \dots, \s_k; \bhp) ( \bI_m^{\otimes^{k-1}} \otimes \bB_{p_j} (\bhp) \bb ) \\
	& = \frac{\partial}{\partial p_j}\left(  \bc^\top \bH_k ( \s_1, \dots, \s_k; \bhp ) ( \bI_m^{\otimes^{k-1}} \otimes \bb ) \right) 
	\end{align*} 
	as desired.  \hfill $\square$

\end{proof}

We now present the final theoretical result,  showing the interpolation of the parameter Hessian.
As the expressions become too involved for a general subsystem transfer function, we write and proof the conditions for the first and second subsystems only, but the results can  
be generalized similarly.


\begin{theorem} \label{thm:hess}
Assume the hypotheses of Corollary \ref{remek:order} for $ q = 2 $.
Define
	\begin{align*}
	\bV_1 (\bp) 	& = \left[ \cA ( \s_1; \bp )^{-1} \bB (\bp) \bb, ~~~ \cA ( \s_2; \bp )^{-1} \bB (\bp) \bb \right], \\
	\bV_2 (\bp)	& = \left[ \cA ( \s_1; \bp )^{-1} \bN (\bp) ( \bI_m \otimes \bV_1 (\bp) ), ~~~ \cA ( \s_2; \bp )^{-1} \bN (\bp) ( \bI_m \otimes \bV_1 (\bp) ) \right] , 		\\
	\bW_1 (\bp) & = \left[ \left( \cA ( \s_1; \bp ) \right)^{-\top} \bC (\bp)^\top \bc, ~~~ \left( \cA ( \s_2; \bp ) \right)^{-\top} \bC (\bp)^\top \bc \right], \\
	\bW_2 (\bp)	& = \left[ \left( \cA ( \s_1; \bp ) \right)^{-\top} \overline{ \bN} (\bp)^\top ( \bI_m \otimes \bW_1 (\bp) ),  ~~~ 
				\left( \cA ( \s_2; \bp ) \right)^{-\top} \overline{ \bN} (\bp)^\top ( \bI_m \otimes \bW_1 (\bp) ) \right] .
	\end{align*}
Assume
	\begin{equation} \label{Vh1}
	\bigcup_{k=1}^2 \bV_k (\bhp) \subseteq \textup{Ran} ( \bV )  ,
	\quad and \quad
	\bigcup_{k=1}^2 \bW_k (\bhp) \subseteq \textup{Ran} ( \bW ) .
	\end{equation}
If either
	\begin{equation}\label{Vh2}
	\bigcup_{k=1}^2 \bigcup_{j=1}^\nu \frac{\partial}{\partial p_j} \bV_k (\bhp) \subseteq \textup{Ran} ( \bV ) 
	\quad \mbox{or} \quad
	\bigcup_{k=1}^2 \bigcup_{j=1}^\nu \frac{\partial}{\partial p_j} \bW_k (\bhp) \subseteq \textup{Ran} ( \bW ) ,
	\end{equation}
then
	\[ 
	\cbH_{\bp} \left( \bc^\top \bH_2 ( \s_1, \s_2; \bhp ) ( \bI_m \otimes \bb ) \right) = \cbH_{\bp} \left( \bc^\top \bHr_2 ( \s_1, \s_2; \bhp ) ( \bI_m \otimes \bb ) \right).
	\]
where $\cbH_{\bp}(\cdot)$ denotes the Hessian with respect to $\bp$.

\end{theorem}

\begin{proof}
Assume we have the extra conditions on $ \bV $.
First note that we have interpolation of $ \cbH_\bp \left( \bc^\top  \bH_1 (\s_i;\bhp) \bb \right) $ for $ i \in \{ 1, 2 \} $ since this is the linear case (see \cite{baur2011interpolatory} for details).
Let $ p_i $ and $ p_j $ refer to any entries in the parameter vector $ \bp $.
Recall the definition of the second transfer function of  the full and reduced model, respectively,
	\begin{align*}
	\bH_2 (s_1,s_2;\bp) & = \bC (\bp) \cA (s_1;\bp)^{-1} \bN (\bp) ( \bI_m \otimes \cA (s_1;\bp)^{-1} \bB (\bp) ) , \\
	\bHr_2 (s_1,s_2;\bp) & = \bCr (\bp) \cAr (s_1;\bp)^{-1} \bNr (\bp) ( \bI_m \otimes \cAr (s_1;\bp)^{-1} \bBr (\bp) ) .
	\end{align*}
We take the second partial derivative of $\bHr_2 (s_1,s_2;\bp)$  with respect to $ p_j $ and $ p_i $, apply the definition of the reduced order matrices, rearrange the terms and use the notation in previous proofs to obtain 
	\begin{align}
	 \frac{\partial^2}{\partial p_j \partial p_i}&\left( \bc^\top \bHr_2 ( \s_1, \s_2; \bhp ) ( \bI_m \otimes \bb ) \right) \nonumber\\
	& = \bc^\top \bC_{p_jp_i} (\bhp) \bV \bffr_2 (\s_1,\s_2;\bhp)\label{m1} \\ 
	& \qquad + \bgr_2^\top (\s_1,\s_2;\bhp) \bW^\top \bB_{p_jp_i} (\bhp) \bb \\
	& \qquad + \bgr_1^\top (\s_2;\bhp) \bW^\top \bN_{p_jp_i} (\bhp) ( \bI_m \otimes \bV \bffr_1 (\s_1;\bhp) ) \\
	& \qquad + \bgr_1^\top (\s_2;\bhp) \bW^\top \cA_{p_jp_i} (\s_2;\bhp) \bV \bffr_2 (\s_1,\s_2;\bhp) \\
	& \qquad + \bgr_2^\top (\s_1,\s_2;\bhp) ( \bI_m \otimes \bW^\top \cA_{p_jp_i} (\s_1;\bhp) \bV \bffr_1 (\s_1;\bhp) ) \label{m2}\\ \label{m3}
	& \qquad + \{ \bgr_1^\top (\s_2;\bhp) \bW^\top \bN_{p_j} (\bhp) + \bgr_2^\top (\s_1,\s_2;\bhp) \bW^\top \cA_{p_j} (\s_1;\bhp) \} ( \bI_m \otimes \bV [ \bffr_1 ]_{p_i} (\s_1;\bhp) ) \\
	& \qquad + \{ \bgr_1^\top (\s_2;\bhp) \bW^\top \bN_{p_i} (\bhp) + \bgr_2^\top (\s_1,\s_2;\bhp) \bW^\top \cA_{p_i} (\s_1;\bhp) \} ( \bI_m \otimes \bV [ \bffr_1 ]_{p_j} (\s_1;\bhp) ) \\
	& \qquad + \{ \bc^\top \bC_{p_j} (\bhp) + \bgr_1^\top (\s_2;\bhp) \bW^\top \cA_{p_j} (\s_2;\bhp) \} \bV [ \bffr_2 ]_{p_i} (\s_1,\s_2;\bhp) \\
	& \qquad + \{ \bc^\top \bC_{p_i} (\bhp) + \bgr_1^\top (\s_2;\bhp) \bW^\top \cA_{p_i} (\s_2;\bhp) \} \bV [ \bffr_2 ]_{p_j} (\s_1,\s_2;\bhp),\label{m4}
	\end{align}
	where $\mathbf{M}_{p_jp_i}(\bp)$ denotes the second partial derivative of $\mathbf{M}(\bp)$ with respect to 
$ p_j $ and $ p_i $, and $[ \bff_k]_{p_i}$ denotes the partial derivative of $\bff_k$ with respect $p_i$.
Then, we follow the similar manipulations used in the proof of Theorem \ref{thm:pbmor2sided} for  \eqref{eq:l1}-\eqref{eq:l4}:
Equations \eqref{m1}-\eqref{m2} contain the same terms, and thus we follow the same reasonings that we used for \eqref{eq:l1}-\eqref{eq:l4}.

Then, even though \eqref{m3}-\eqref{m4} contain the new terms $ [ \bff_1 ]_{p_j} (\s_1;\bhp) $ and $ [ \bff_2 ]_{p_j} (\s_1,\s_2;\bhp) $, the same manipulations still apply here 
due to the construction of $\bV$ in \eqref{Vh1} and  \eqref{Vh2}, 
 $ [ \bff_1 ]_{p_j} (\s_1;\bhp) $ and $ [ \bff_2 ]_{p_j} (\s_1,\s_2;\bhp) $ are now also spanned by $ \textup{Ran} (\bV) $ for any $ p_j $. Therefore, we obtain
 $$
\frac{\partial^2}{\partial p_j \partial p_i}\left( \bc^\top \bHr_2 ( \s_1, \s_2; \bhp ) ( \bI_m \otimes \bb ) \right)	 = 
\frac{\partial^2}{\partial p_j \partial p_i}\left( \bc^\top \bH_2 ( \s_1, \s_2; \bhp ) ( \bI_m \otimes \bb ) \right).
 $$
Since $p_j$ and $p_i$ were arbitrary, we obtain the Hessian matching as desired.
The proof would be analogous if we assumed the extra conditions on $ \bW $ instead. 
Only the rearrangement of the terms would change so that the expression depends on $ [ \bg_1 ]_{p_j} (\s_2;\bhp) $ and $ [ \bg_2 ]_{p_j} (\s_1,\s_2;\bhp) $ instead.
\hfill $\square$

\end{proof}

\begin{remark}
As we stated above, one can write the conditions for matching the parameter Hessian of the higher index  subsystems. 
Let $ q $ be the number of subsystems we wish to interpolate.
To obtain the parameter Hessian matching for the general $k^{\rm th}$ order subsystem, i.e., to 
satisfy
	\[ 
	\cbH_\bp \left( \bc^\top \bH_k (\s_1, \dots, \s_k; \bhp) ( \bI_m^{\otimes^{k-1}} \otimes \bb ) \right)
	= \cbH_\bp \left(\bc^\top \bHr_k (\s_1, \dots, \s_k; \bhp) ( \bI_m^{\otimes^{k-1}} \otimes \bb ) \right),
	\qquad k = 1, \dots, q
	\]
we would need  
	\begin{align*}
	\bV_1 (\bp) 	& = \left[ \cA ( \s_1; \bp )^{-1} \bB (\bp) \bb, ~~~ \cdots, ~~~ \cA ( \s_q; \bp )^{-1} \bB (\bp) \bb \right], \\
	\bV_k (\bp)	& = \left[ \cA ( \s_1; \bp )^{-1} \bN (\bp) ( \bI_m \otimes \bV_{k-1} (\bp) ), ~~~\cdots,~~~ \cA ( \s_q; \bp )^{-1} \bN (\bp) ( \bI_m \otimes \bV_{k-1} (\bp) ) \right] , & k=1,\dots,q 		\\		
	\bW_1 (\bp) & = \left[ \left( \cA ( \s_1; \bp ) \right)^{-\top} \bC (\bp)^\top \bc, ~~~ \cdots, ~~~ \left( \cA ( \s_q; \bp ) \right)^{-\top} \bC (\bp)^\top \bc \right], \\
	\bW_k (\bp)	& = \left[ \left( \cA ( \s_1; \bp ) \right)^{-\top} \overline{ \bN} (\bp)^\top ( \bI_m \otimes \bW_1 (\bp) ), ~~~\cdots, ~~~ \left( \cA ( \s_q; \bp ) \right)^{-\top} \overline{ \bN} (\bp)^\top ( \bI_m \otimes \bW_1 (\bp) ) \right] & k=1,\dots, q.
	\end{align*}
(evaluated at $ \bhp $) to be contained in the ranges of the basis $ \bV $ and $ \bW $, respectively,
\emph{together with} either the partial derivatives of the $ \bV_k $'s or the partial derivatives of the $ \bW_k $'s with respect to the parameter entries (evaluated at $ \bhp $).
\end{remark}


\section{Numerical Examples} 
	\label{examples}

In this section, we illustrate the theoretical discussion from Section \ref{problem} using two examples: A nonlinear RC circuit in Section \ref{ex:rc} and an advection-diffusion equation in Section \ref{ex:heat}. Throughout this section, $\bhp^{(i)}$ (or  $\hat{p}^{(i)}$ when the parameter is a scalar) denotes the parameter sampling points we used in constructing the model reduction bases $\bV$ and $\bW$, and $\bp^{(i)}$ (or  ${p}^{(i)}$) denotes the parameter points (which are not sampled) at which we evaluate both reduced and full models to investigate the accuracy of the reduced model. 

\subsection{A nonlinear RC circuit} \label{ex:rc}

We begin with a modified version of a standard benchmark problem for bilinear systems, namely a nonlinear RC circuit \cite{bai2006projection,phillips2003projection}. The original benchmark problem leads to a non-parametric bilinear system. We have revised the problem to add parametric dependence. 
 To clearly motivate this parametric dependence, we include details of the model derivation.

Consider the following SISO parametric nonlinear system
	\begin{align}
	\label{RC}
	\left\{ \begin{array}{l}
		\dot{\bv} (t;p) = \bff ( \bv (t) ; p ) + \bb u (t) \\
		y (t;p) = \bc^\top \bv (t;p),
	\end{array} \right.
	\end{align}
where $ \bv (t;p) \in \R^n $, $ \bb = \bc = \left[ 1 \ 0 \ \cdots \ 0 \right]^\top  \in \R^n$, 
	\begin{align}
	\bff ( \bv ; p ) 
	= \left[ \begin{array}{c} 
		- g ( v_1 ; p ) - g ( v_1-v_2 ; p ) \\
		g ( v_1-v_2 ; p ) - g ( v_2-v_3 ; p ) \\
		\vdots \\
		g ( v_{k-1}-v_k ; p ) - g ( v_k-v_{k+1} ; p ) \\
		\vdots \\
		g ( v_{N-1}-v_N ; p ) 
		 \end{array} \right] ,
	\end{align}
and
	\begin{align}
	g ( v ; p ) = e^{ p v } + v - 1. 
	\end{align}
System \eqref{RC} models a nonlinear RC circuit with $ N $ resistors where the state variable $ \bv (t;p) $ is the voltage at each node,  $ u (t) $ is the input signal to the current source, the ouput $ y (t;p) $ is the voltage between node 1 and ground, and $ g (\nu;p) $ gives the current-voltage dependency at each resistor.
We have introduced a parameter dependency $ p \in \R $ in the exponential term of this current-voltage dependency, which models the influence of the operating temperature on the current.
Following \cite{bai2006projection,phillips2003projection}, we  apply Carleman bilinearization to $ \bff ( \bv ; p ) \approx \bA_1 (p) \bv + \bA_2 (p) ( \bv \otimes \bv ) $ and a second-order approximation of $ g ( v ; p ) \approx (p+1)v + \frac{1}{2} p^2 v^2 $ leading to an approximation of the nonlinear dynamics \eqref{RC} by the following parametric bilinear  system:
	\begin{align}
	\label{RCb}
	\left\{ \begin{array}{l}
		\bE \dot{\bx} (t;p)   = \bA (p) \bx (t;p) + \bN \bx (t;p) u (t) + \bb u (t)	\\
		y_b (t;p)			= \bc^\top \bx (t;p),
	\end{array} \right.
	\end{align}
where
	\begin{align}
		\bx (t;p) &= \left[ \begin{array}{c} \bv (t;p) \\ \bv (t;p) \otimes \bv (t;p) \end{array} \right],\quad\quad\quad\quad\bE = \bI_n,
	& 	\bc & = \bb = \left[ \begin{array}{c} 1 \\ \bf 0 \end{array} \right], \\
	\bA (p) & = \left[ \begin{array}{cc} \bA_1 (p) & \bA_2 (p) \\ \bf 0 & \bA_1 (p) \otimes \bI + \bI \otimes \bA_1 (p) \end{array} \right],
	&	\bN & = \left[ \begin{array}{cc} \bf 0 & ~~ \bf 0 \\ \bb \otimes \bI + \bI \otimes \bb  & ~~ \bf 0 \end{array} \right],
	\end{align}
and we use $y_b(t;p)$ to denote the output of the full-order parametric bilinear system.
Note that the dimension of the bilinear system is $ n = N + N^2 $ and the matrices $ \bA_1 (p) \in \R^{N\times N} $ and $ \bA_2 (p) \in \R^{N\times N^2} $ are given by
	\begin{align*}
	\bA_1 (p) 
	& = (1 + p ) \left[ \begin{array}{ccccc}
	-2 & 1 &  \\
	1 & -2 & 1  \\
	& \ddots & \ddots & \ddots \\
	& & 1 & -2 & 1 \\
	& & & 1 & - 1
	\end{array} \right],
	\end{align*}
and, for $ k = 2, \dots, N - 1 $,
	\begin{align*}
	[\bA_2(p)]_{(1,1)} & = -p^2 , \\
	[\bA_2(p)]_{(1,2)} & = [\bA_2(p)]_{(1,N+1)} = [\bA_2(p)]_{(k,(k-2)N + k - 1)} = [\bA_2(p)]_{(k,(k-1)N + k + 1)} = \\ 
		& = [\bA_2(p)]_{(k,kN + k)} = [\bA_2(p)]_{(N,(N-2)N + N - 1)} = [\bA_2(p)]_{(N,(N-1)N+ N)} = \frac{p^2}{2} , \\
	[\bA_2(p)]_{(1,N+2)} & = [\bA_2(p)]_{(k,(k-2)N + k)} = [\bA_2(p)]_{(k,(k-1)N + k - 1)} = [\bA_2(p)]_{(k,kN + k + 1)} = \\
		& = [\bA_2(p)]_{(N,(N-2)N + N )} = [\bA_2(p)]_{(N,(N-1)N + N-1)} = - \frac{p^2}{2},
	\end{align*}
	where $[\bA_2(p)]_{(i,j)}$ denotes the $(i,j)^{\rm th}$ entry of $\bA_2(p)$.
Note that both $\bA_1(p)$ and $\bA_2(p)$ have the desired  affine structure \eqref{eq:affine} with the nonlinear scalar parametric functions $-p^2$, $\frac{1}{p^2}$, and $-\frac{1}{p^2}$.
 As in the original benchmark problem, we choose $ N = 200 $, and thus obtain a parametric bilinear system of dimension $ n= 40,200$. \textcolor{black}{We are interested in the parameter range $p \in [0,70]$}, and choose two parameter sampling points, $\hat{p}^{(1)} = 1$ and $\hat{p}^{(2)} = 50$. For each sampling point, we focus on the leading $q=2$ subsystems. We choose
$ \{ \sigma_1, \sigma_2 \}$ by running IRKA on the linearized model (by setting $\bN=0$); i.e,
  $\{ \sigma_1, \sigma_2 \} $ correspond to optimal sampling points for the linear model.
With these frequencies, we construct the basis $ \bV_1 $ and $ \bW_1$ (using Theorem {\color{black}\ref{thm:hess}})  that guarantees interpolation of $ \bH_1 (s;p) $, $ \bH_2 (s_1,s_2;p) $, and their sensitives for $ p=\hat{p}^{(1)} = 1,$ and at $ \{ \sigma_1, \sigma_2 \}$.  Similarly, we construct $ \bV_2 $ and $ \bW_2 $ for $ \hat{p}^{(2)} = 50 $. We then construct the global bases $ \bV = [ \bV_1 \ \bV_2 ] $ and $ \bW = [ \bW_1 \ \bW_2 ]$ and obtain a reduced parametric bilinear model of dimension $ r = 12 $ using the projection described in \eqref{eq:romss}; thus we are approximating a parametric bilinear system of dimension 
$n=40,200$ by a reduced parametric bilinear model of dimension $r=12$. To test the accuracy of the parametric reduced model, we simulate and compare the outputs of the original nonlinear model  \eqref{RC}, the full bilinear model \eqref{RCb}, and the reduced bilinear model for two different inputs, $u(t) =  e^{-t} $
and $u(t) = \frac{1}{2} ( \cos ( 5 \pi t) + 1 ) $, and three different parameter values, $p^{(1)} = 18$,
$p^{(2)} = 40$, and $p^{(3)}=62$. Note that these parameter values are not the sampled values; indeed 
 $p_3=62$ is even outside the sampling range $[1,50]$. Moreover, note that the inputs $u(t) =  e^{-t} $
and $u(t) = \frac{1}{2} ( \cos ( 5 \pi t) + 1 ) $ were not used in the model reduction step, i.e., the reduced model is not informed by these choices of excitation. As  Figure \ref{figRC} shows, the parametric reduced bilinear system provides a very accurate approximation to the full bilinear model; their responses are almost indistinguishable. 
Relative $L_2$ errors in the outputs for our three parameter values $ p^{(1)} $, $ p^{(2)} $, $ p^{(3)} $ are listed in Table \ref{tab:rc}, showing a relative error on the order of $10^{-3}$. We also emphasize that the only deviations visible in Figure \ref{figRC} are deviations from the original nonlinear system, due to  Carleman bilinearization, and  not due to the model reduction step. We also note that 
even though the responses might look similar for different parameter values, the scales of the outputs are different.
	\begin{table}
	\centering
	\begin{tabular}{|c||c|c|c|} 
	\hline\noalign{\smallskip}
	\mbox{Input}			& $ p=p^{(1)} $ 	& $ p=p^{(2)} $ 		& $ p=p^{(3)} $	 \\
	\noalign{\smallskip} \hline \noalign{\smallskip} 
	$u(t) = e^{-t}$	& $ 2.54 \times 10^{-3} $	& $ 2.91 \times 10^{-3} $	& $ 1.53 \times 10^{-3} $	 	\\ \hline
	\noalign{\smallskip} $u(t) =   \frac{1}{2} ( \cos ( 5 \pi t) + 1 ) $	
				& $ 2.54 \times 10^{-3} $	& $ 4.33 \times 10^{-3} $	& $ 4.57 \times 10^{-3} $		  \\
	\noalign{\smallskip}\hline
	\end{tabular}
		\caption{Relative $L_2$ output error }
			\label{tab:rc}
	\end{table}
To make the numerical investigations more detailed, we performed a parameter sweep using $10^3$ \textcolor{black}{linearly} sampled points
in the interval \textcolor{black}{$[0,70]$}, then identified the performance of the reduced model measured in terms of the relative $L_2$ error in the output $y_b(t;p)$ for each of the inputs. 
We found that for the first input $u(t) = e^{-t}$, in the  \emph{worst}-case scenario, the reduced model led to 
 a relative $L_2$ output error of {\color{black}$6.31\times10^{-3}$}. 
For the second input $u(t) = \frac{1}{2} ( \cos ( 5 \pi t) + 1 )$, the \emph{worst} 
performance yielded a relative $L_2$ output error of {\color{black}$5.96\times10^{-3}$}.
 These numbers further illustrate the ability of the reduced model to accurately approximate the full order parametric bilinear system. 

	\begin{figure}
	
	\includegraphics*[width=0.45\textwidth]{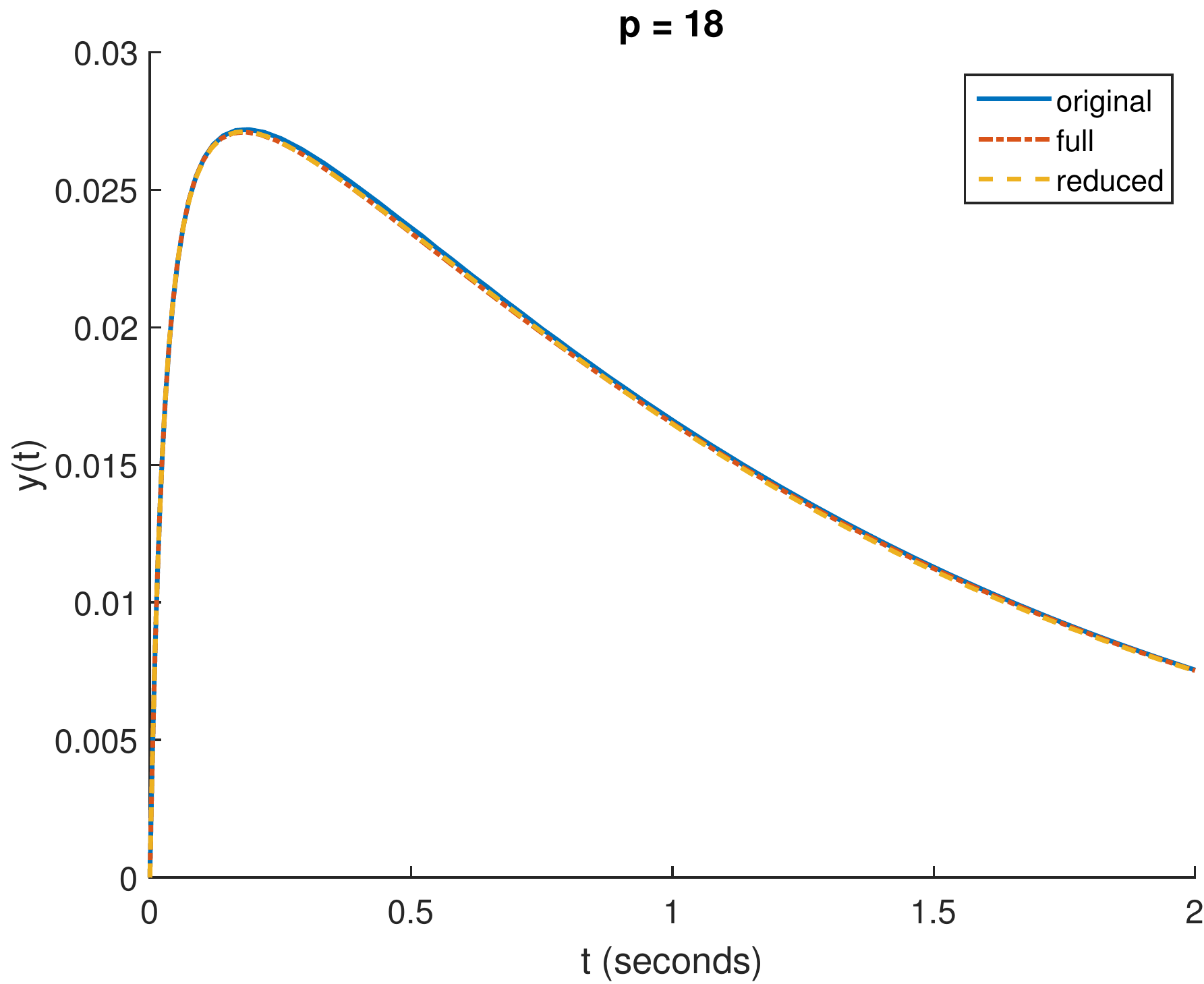} \ 
	\includegraphics*[width=0.45\textwidth]{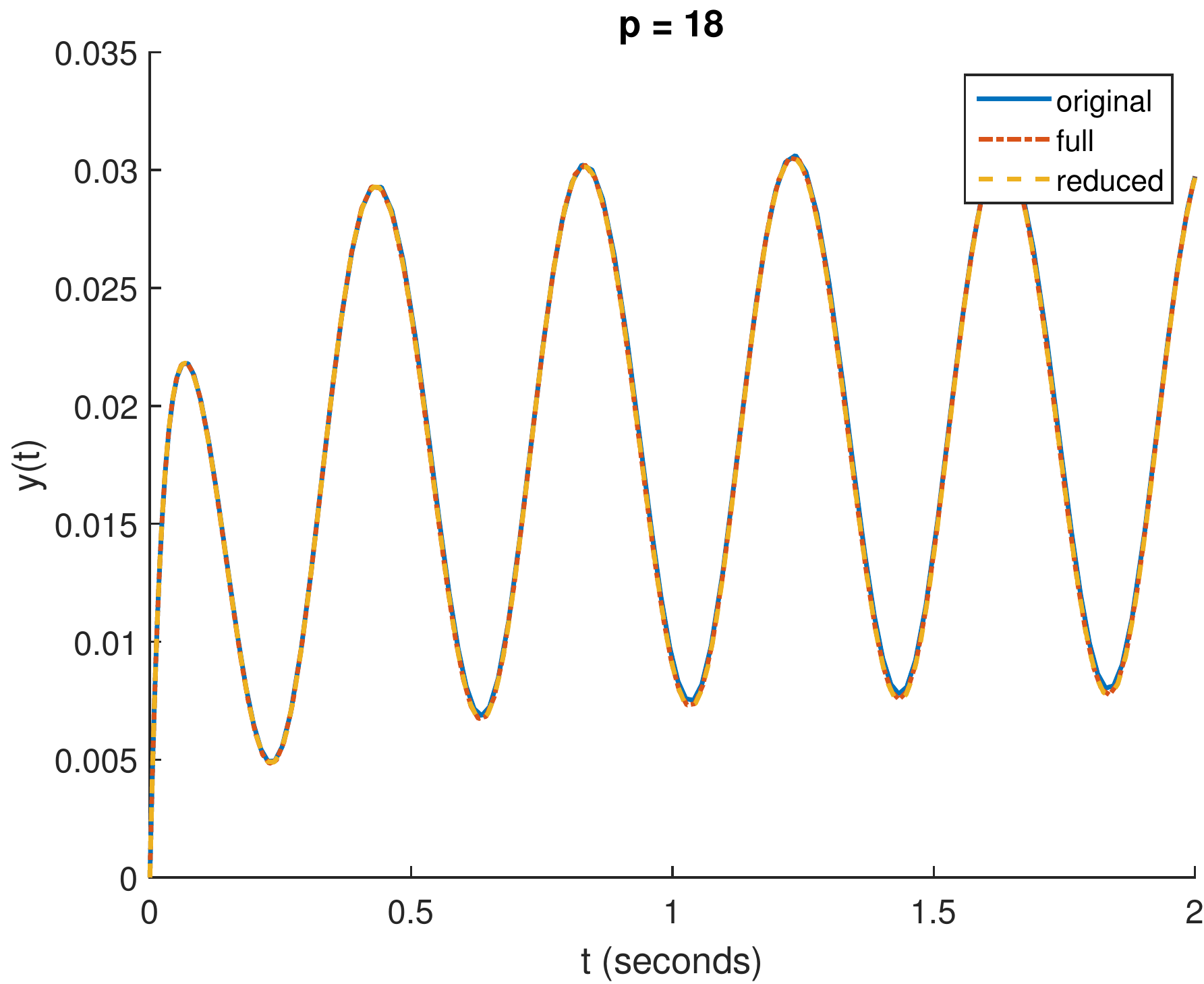} 

	\includegraphics*[width=0.45\textwidth]{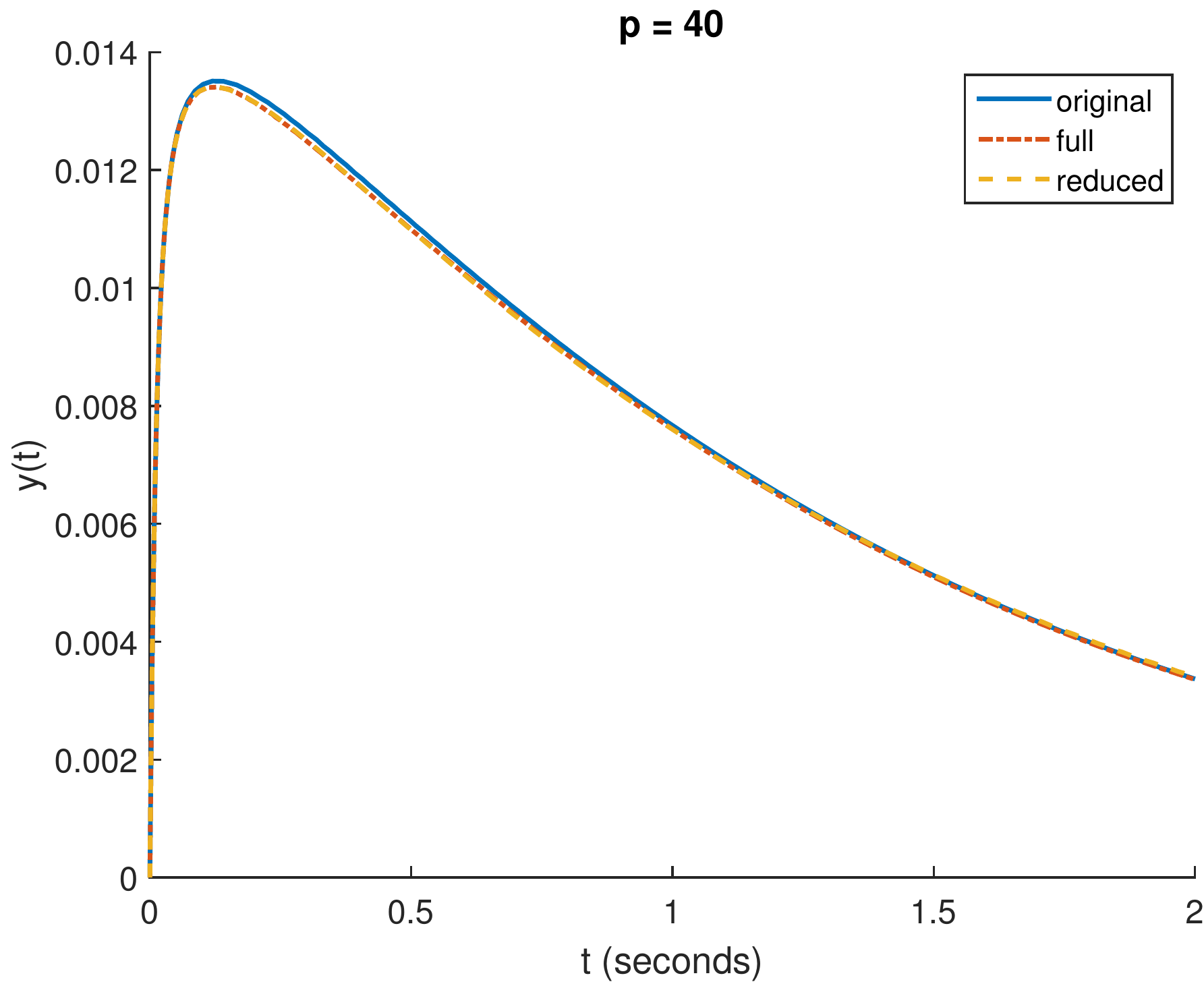} \ 
	\includegraphics*[width=0.45\textwidth]{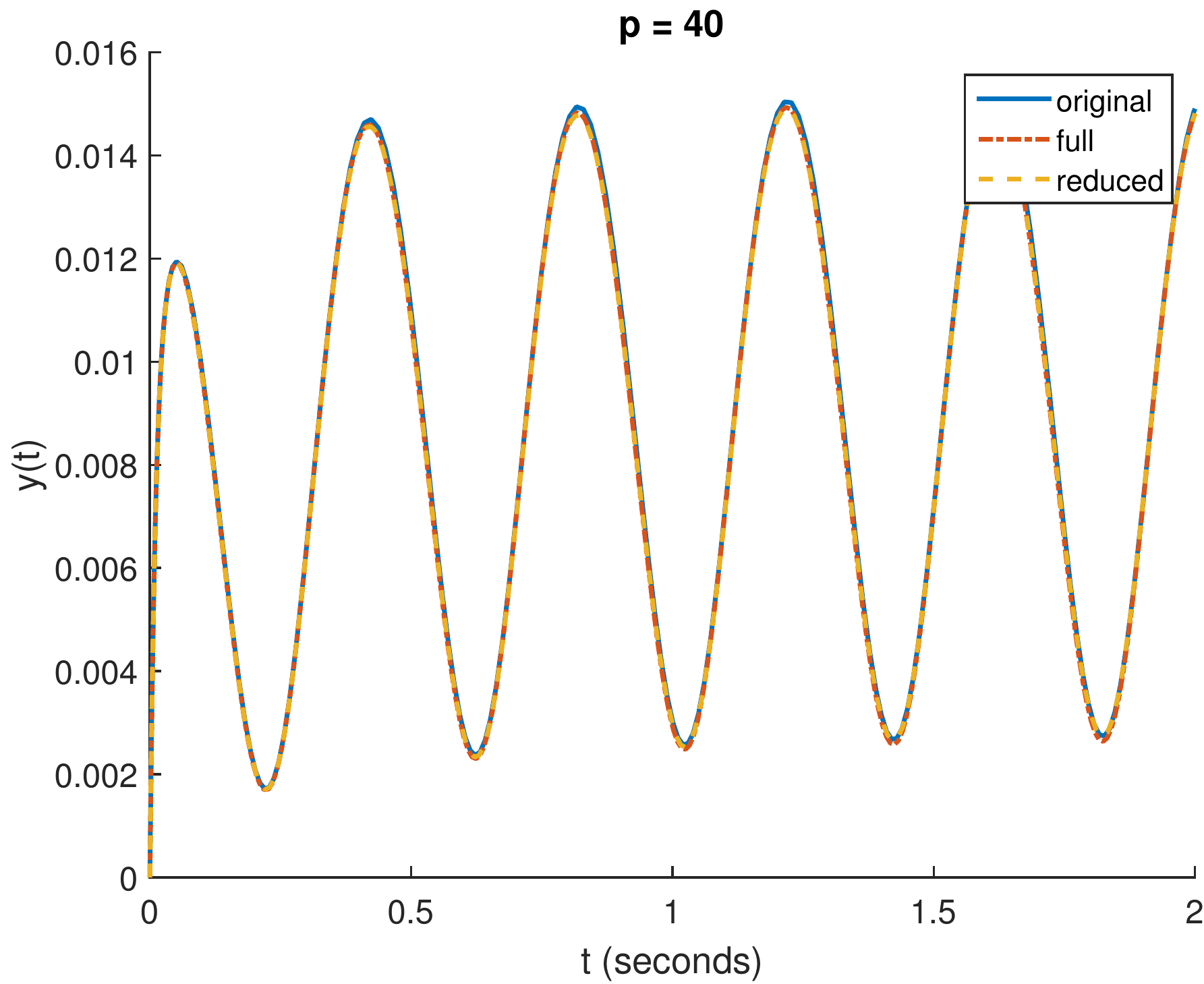} 

	\includegraphics*[width=0.45\textwidth]{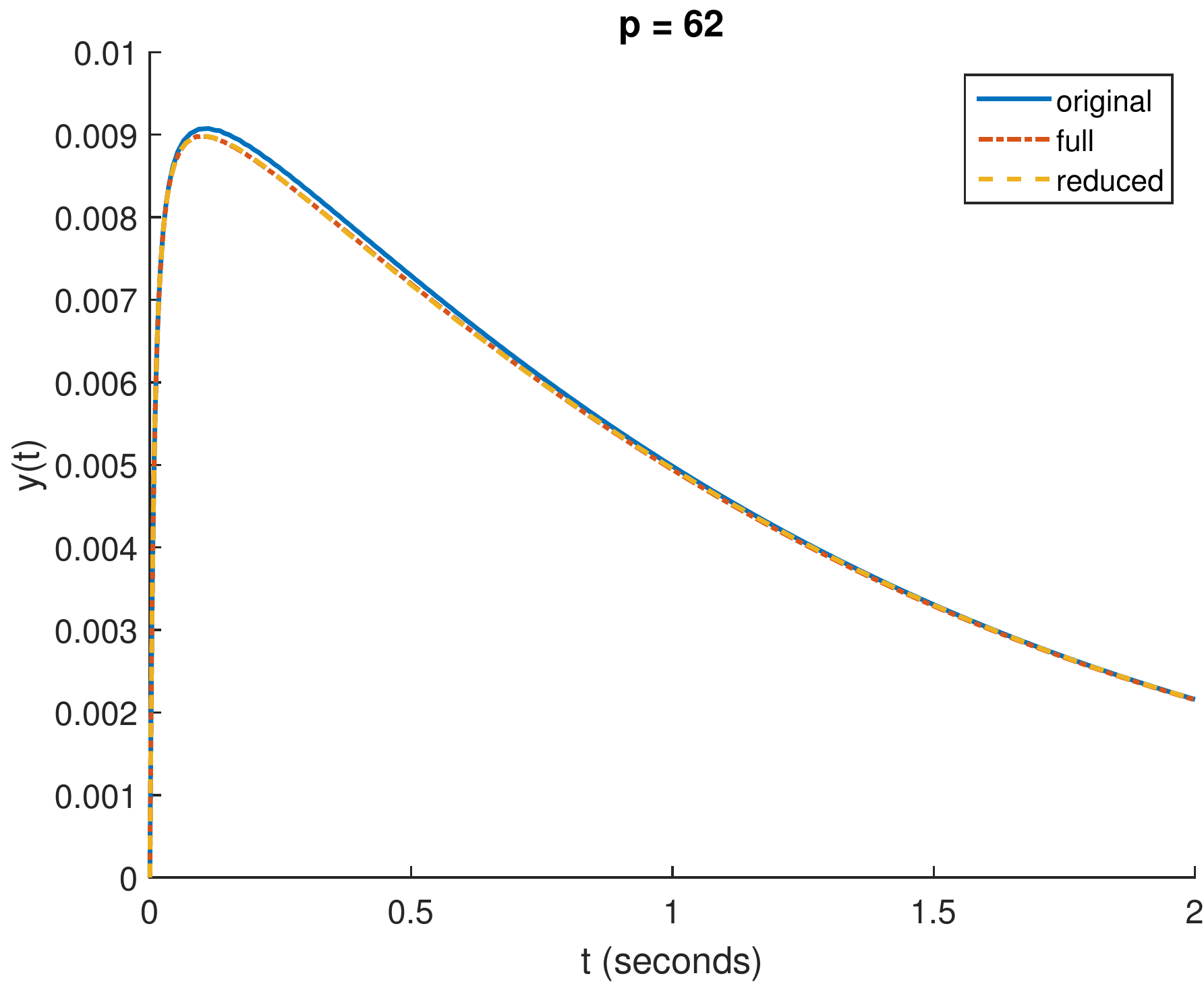} \ 
	\includegraphics*[width=0.45\textwidth]{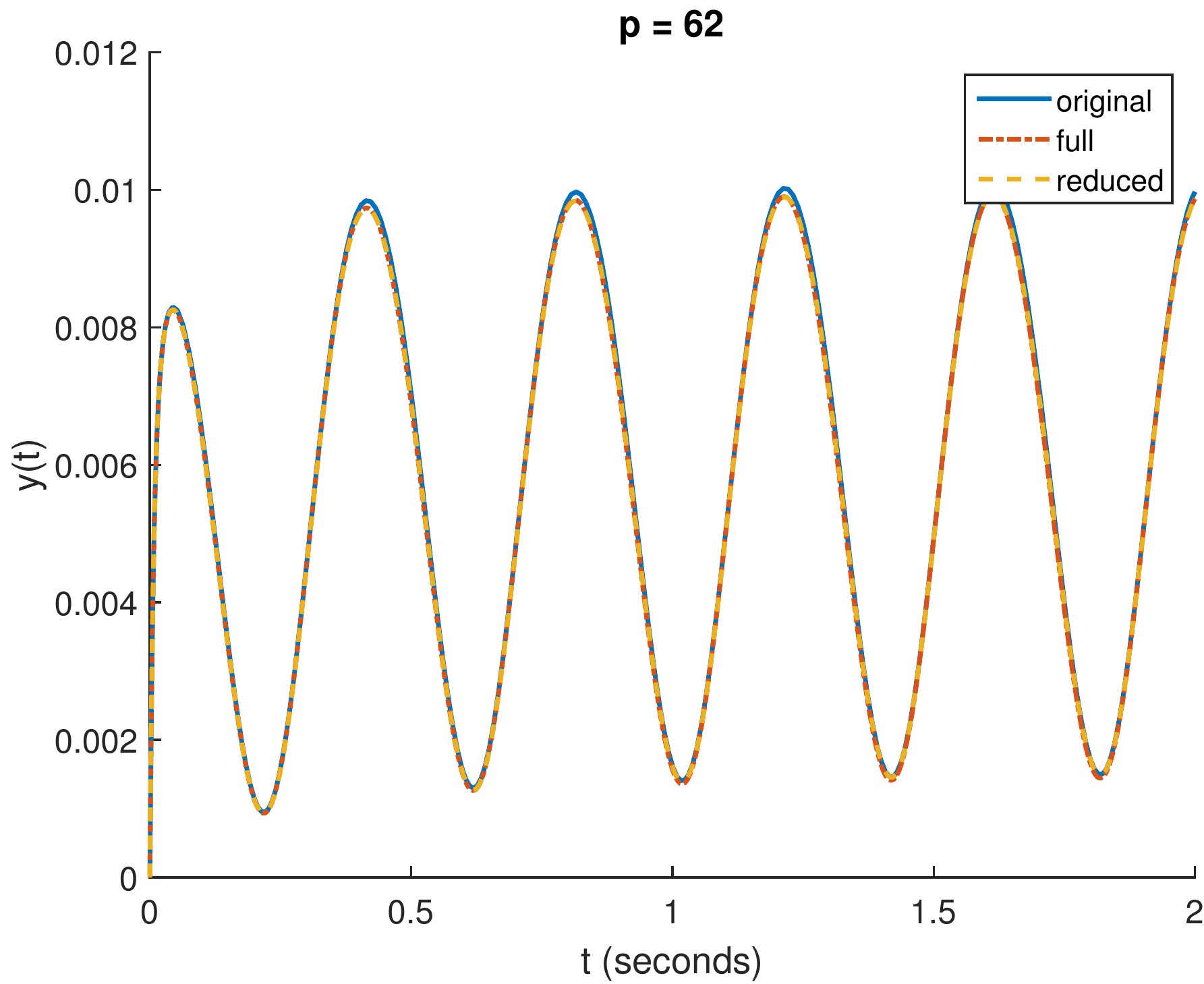} 

	\caption{Solution to \eqref{RC} (denoted by ``original"), \eqref{RCb} (denoted by ``full"), and reduced order model (denoted by ``reduced") for different inputs and parameter values. 
		Left column: $ u (t) = e^{-t} $. 
		Right column: $ u (t) = \frac{1}{2} ( \cos ( 5 \pi t) + 1 ) $.}
	\label{figRC}
	\end{figure}

\subsection{Advection-diffusion equation} \label{ex:heat}

For our second example consider a model of the transport and diffusion of a passive scalar field $T$ (representing a chemical concentration, temperature, etc.) on the domain $ \Omega = [-1,1] \times [-1,1] $. The transport of $T$ is controlled using a background velocity field described with two input parameters $u_1$ and $u_2$, and two velocity fields $\bv_1$ and $\bv_2$.  Thus the background velocity field is 
	$ \bv (x,y) = u_1 (t) \bv_1 (x,y) + u_2 (t) \bv_2 (x,y) $.   The value of the passive scalar on the boundary of $\Omega$ ($\partial\Omega$) is controlled by an input $u_3$.
We model the diffusion using the viscosity parameter $ p_1 $, 
and include a source term centered at $ (p_2, p_3) \in \Omega $ with an area of affect described by $p_4$ given by
	\[ f (x,y;p_2,p_3,p_4) = \exp \left( - \frac{ ( x - p_2 )^2 + ( y - p_3 )^2 }{ p_4 } \right). \]
The strength of the source term is controlled by an input $u_4$.

Our passive scalar field $ T $ then satisfies 
	\begin{align}
	\label{eq:advectionDiffusion}
	\dot{T}(x,y,t) 	& = p_1 \Delta T(x,y,t) - \bv \cdot \nabla T(x,y,t) + u_4(t) f(x,y;p_2,p_3,p_4), 	& 
	(x,y,t) 	& \in \Omega \times (0,\infty) ,				\\
	T (x,y, 0) 	& = T_0 (x,y), 							& 
	(x,y) 		& \in \Omega,							\\
	T (x,y,t) 	& = u_3(t) ,								& 
	(x,y,t) 	& \in \partial \Omega \times (0,\infty).
	\end{align}
Thus our model depends on the parameter vector
	\[ \bp = \left[ \begin{array}{c} p_1 \\ p_2 \\ p_3 \\ p_4 \end{array} \right] . \]
We will consider the following as our parameter range
	\[ -3 \le \ln p_1 \le 1, \qquad (p_2, p_3) \in \Omega, \qquad 1 \le p_4 \le 10.  \]
We approximate solutions to (\ref{eq:advectionDiffusion}) using a finite element discretization $ T_N (x,y,t) = \sum_{j=1}^N x_j (t) \varphi_j (x,y) $, where the $\{\varphi_j\}_{j=1}^N$ arise from quadratic (P2) triangular elements.  For convenience, we will split the summation above into two disjoint parts, one with indices corresponding to boundary nodes (${\cal B}$) and the remainder corresponding to interior nodes (${\cal I}$).  Thus
$\{ 1, 2, \ldots, N \} = {\cal B} \cup {\cal I}$. Upon substituting this into the weak form of (\ref{eq:advectionDiffusion}) and suppressing function arguments, we arrive at
	\begin{displaymath}
	\left( \dot{ \left[ \sum_{j=1}^N x_j \varphi_j \right] } , \varphi_i \right) 
	 = -\left( p_1 \nabla \left[ \sum_{j=1}^N x_j \varphi_j \right] , \nabla\varphi_i \right)
	- \left( \bv \cdot \nabla \left[ \sum_{j=1}^N x_j \varphi_j \right] , \varphi_i \right)
	+ \left( u_4 f , \varphi_i \right), \qquad \forall i\in {\cal I},
	\end{displaymath}
where the boundary integrals vanish since $\varphi_i$ are zero on the boundary when $i\in {\cal I}$.  Interchanging integration in the $L_2$-inner products with the summation leads to
	\begin{align*}
	\sum_{j\in{\cal I}}  \left( \varphi_j , \varphi_i \right) \dot{ x }_j
	& = - p_1 \sum_{j\in{\cal I}} \left( \nabla \varphi_j  , \nabla \varphi_i \right) x_j 
	- p_1 u_3 \sum_{j\in{\cal B}} \left( \nabla \varphi_j  , \nabla \varphi_i \right) 
	\\
	& \quad - u_1 \sum_{j\in{\cal I}}  \left( \bv_1 \cdot \nabla \varphi_j , \varphi_i \right) x_j
	- u_2 \sum_{j\in{\cal I}} \left( \bv_2 \cdot \nabla \varphi_j , \varphi_i \right) x_j
	+ u_4 \left(  f , \varphi_i \right),
	\end{align*}
for each $i\in{\cal I}$.
Letting $ \bx (t) = [x_1(t) \ x_2(t) \ \dots \ x_{|{\cal I}|}(t)]^\top $ and
	\begin{align*}
	[ \bE ]_{ij} 		& = \left( \varphi_i , \varphi_j \right) ,
	& [ \bN_1 ]_{ij} 	& = - \left( \bv_1 \cdot \nabla \varphi_j , \varphi_i \right) ,
	& [ \bb_3 ]_i		& = - p_1 \sum_{k\in{\cal B}} \left( \nabla \varphi_i  , \nabla \varphi_k \right),
	\\
	[ \bA ]_{ij} 		& = -p_1 \left( \nabla \varphi_i  , \nabla \varphi_j \right) ,
	& [ \bN_2 ]_{ij} 	& = - \left( \bv_2 \cdot \nabla \varphi_j , \varphi_i \right) ,
	& [ \bb_4 ]_i		& = \left( f(\cdot;p_2,p_3,p_4) , \varphi_i(\cdot) \right),
	\end{align*}
for $i,j\in{\cal I}$. We can write our discrete problem as
	\begin{align*}
	\bE \dot{\bx} (t;p) = \bA (\bp) \bx (t;p) + \bN_1 \bx (t;p) u_1 (t)  + \bN_2 \bx (t;p) u_2 (t) + \bb_3 (\bp) u_3 (t) + \bb_4 (\bp) u_4 (t) ,
	\end{align*}
or
	\begin{align*}
	\bE \dot{\bx} (t;p) = \bA (\bp) \bx (t;p) + \sum_{i=1}^4 \bN_i \bx (t;p) u_i (t)  + \bB (\bp) \bu (t),
	\end{align*}
where
	\begin{align*}
	\bN & = [ \bN_1 \ \bN_2 \ \bN_3 \ \bN_4 ] = [ \bN_1 \ \bN_2 \ \textbf{0} \ \textbf{0} ] , \\
	\bB (\bp) & = [ \textbf{0} \ \textbf{0} \ \bb_3 (\bp) \ \bb_4 (\bp) ] , ~~\mbox{and}\\
	\bu (t) & = [ u_1 (t) \ u_2 (t) \ u_3 (t) \ u_4 (t)  ]^\top . 
	\end{align*}
We  also include an output $ \by (t;p) = \bc^\top \bx (t;p) $ that represents the average of our scalar field 
over $ [0.5,1] \times [0.5,1] $.
In summary, we have a bilinear parametric  multi-input/single-output system
	\begin{align}
	\label{Heatb}
	\left\{ \begin{array}{l}
		\bE \dot{\bx} (t;p) = \bA (\bp) \bx (t;p) + \sum_{i=1}^4 \bN_i \bx (t;p) u_i (t)  + \bB (\bp) \bu (t) 	\\
		y (t;p)			= \bc^\top \bx (t;p),
	\end{array} \right.
	\end{align}
which can be reduced using the strategy presented in the  previous example.

For our simulations, we chose a 21-by-21 FEM mesh (which results in a FOM of dimension $ n = |{\cal I}| =361 $) and velocity with
	$ \bv_1 (x,y) = [ -y, \ x ]^\top $ 
and 
	$ \bv_2 (x,y) = \frac{1}{2} ( \cos ( \pi ( x - y ) ) + 1 ) [ 1, \ 1 ]^\top $.

Note that the full order matrices $ \bE, \bN_1, \bN_2, \bc $ are constant, $ [\bA (\bp)]_{ij} = -p_1 [\bA]_{ij} $, and $ [\bb_3 (\bp)]_i = - p_1 [\bb_3]_i $, hence the ROM is given by
	\begin{align}
	\label{Heatbr}
	\left\{ \begin{array}{l}
		\bEr \dot{\bxr} (t;p) = \bAr (\bp) \bxr (t;p) + \bNr_1 \bxr (t;p) u_1 (t)  + \bNr_2 \bxr (t;p) u_2 (t) + \bbr_3 (\bp) u_3 (t) + \bbr_4 (\bp) u_4 	(t) \\
		\widetilde{y} (t;p)			= \bcr^\top \bxr (t;p)
	\end{array} \right.
	\end{align}
where
	\begin{align}
	\bEr & = \bW^\top \bE \bV 	&	\bAr (\bp) & = - p_1 \bW^\top \bA \bV \\
	\bNr_j & = \bW^\top \bN_j \bV 	&	\bbr_3 (\bp) & = - p_1 \bW^\top \bb_3 \\
	\bcr^\top & = \bc^\top \bV 	&	\bbr_4 (\bp) & = \bW^\top \bb_4 (\bp).
	\end{align}

Even though the dimension in \eqref{Heatbr} is lower, 
the reduction of the vector $ \bb_4 (\bp) $ can not be done offline (as for the rest of the system matrices),
hence we aim to reduce the cost of computing $ \bb_4 (\bp) $ by means of DEIM approximation  as we discussed in Section \ref{sec:projintro}. In this case, since $ \bb_4 (\bp) $ is a vector, there is no need for a matrix-version and the original DEIM formulation suffices.  Applying DEIM, we want to find a basis $ \bU \in \R^{n\times M} $ where $ M \ll n $ and a row selector $ \bbS $ so that 
	\[ \bb_4 (\bp) \approx \bU ( \bbS^\top \bU )^{-1} \bbS^\top \bb_4 (\bp) 
	\quad \mbox{and} \quad \bbr_4 (\bp)  \approx \bW^\top \bU ( \bbS^\top \bU )^{-1} \bbS^\top \bb_4 (\bp)
	\]
are good approximations. 
This way $  \bW^\top\bU ( \bbS^\top \bU )^{-1} $ can be precomputed offline, while the online computation of $ \bbS^\top \bb_4 (\bp) $ will now only require us to compute the entries in $ \bb_4 (\bp) $ indicated by $ \bbS $.
Clearly the accuracy of this approximation depends on $ \bU $ and $ \bbS $.

We first find $ \bU $ by Proper Orthogonal Decomposition (POD) \cite{lumley,berkooz}.  That is, we generate a matrix of snapshots of the vector $ \bb_4 (\bp) $ and  select the leading  $ M $ left singular vectors to be the columns of $ \bU $.
By taking \emph{enough} snapshots and singular vectors, we expect the range of our basis $ \bU $ to represent the values of $ \bb_4 (\bp) $ over the parameter domain.
We chose a tolerance of $ 10^{-5} $ to truncate the singular values in the POD basis,
 resulting in a DEIM approximation of order $M=33$. To chose the interpolation indices (row selector) in $\bbS$,   we use the Q-DEIM algorithm \cite{drmac2016new} which determines $ \bbS $ using a pivoted QR factorization of $\bU^\top$.
	
In Figure \ref{ADRerrorDEIM} we show the relative error of the Q-DEIM approximation of $\bb_4(\bp)$ over $ 10^4 $ random parameter values in the entire parameter domain.
Note that the maximum relative error is on the order of $ 10^{-4} $, showing the accuracy of the DEIM approximation. Thus, we can confidently use $\bbr_4 (\bp)  \approx \bW^\top \bU ( \bbS^\top \bU )^{-1} \bbS^\top \bb_4 (\bp)$ in our reduced model.

	\begin{center} \begin{figure} [h!]
	\centering
	\includegraphics[width=0.45\textwidth]{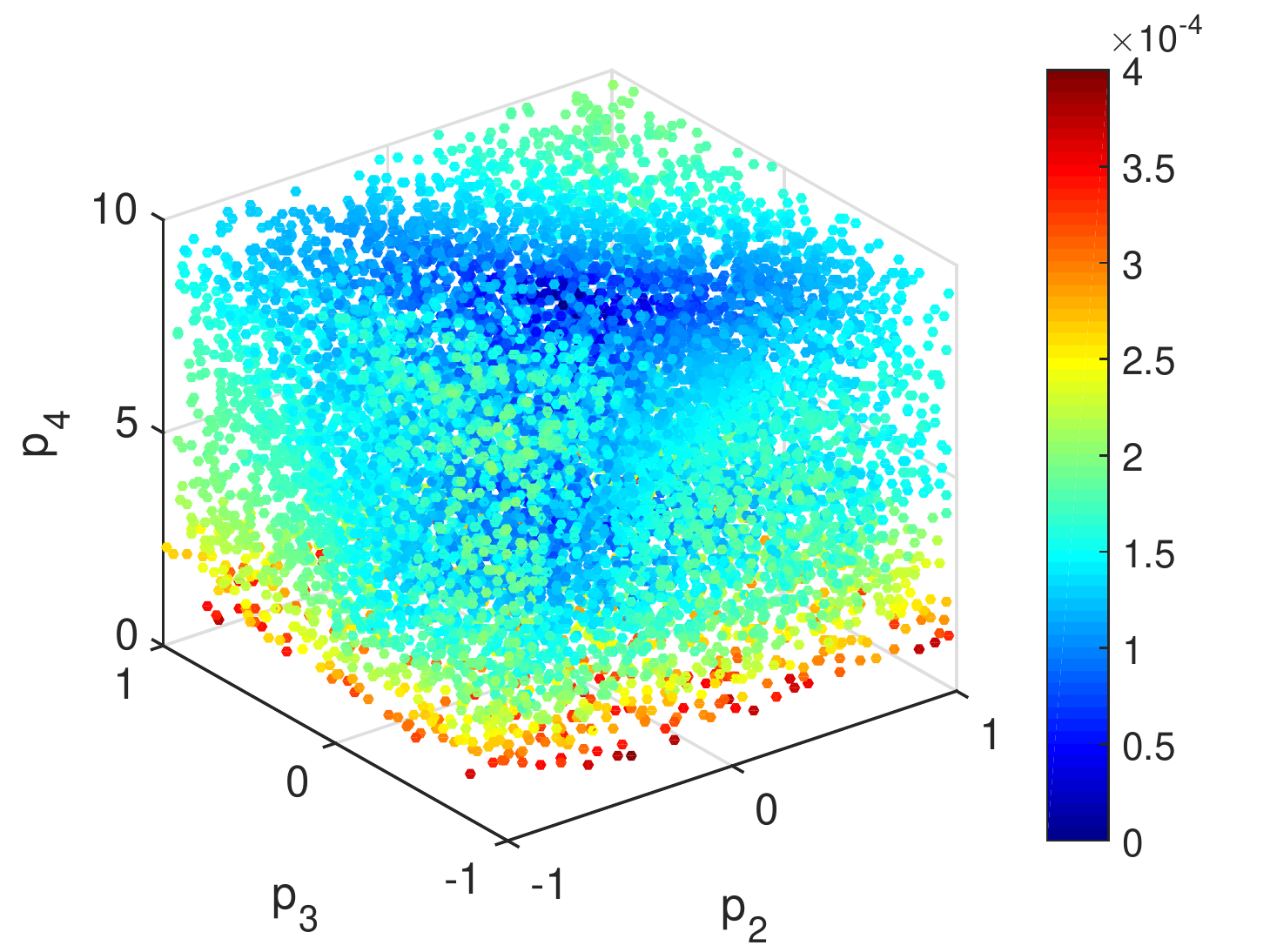}  	
	\caption{Relative error in the DEIM approximation of $ \bb_4  (\bp) $.}
	\label{ADRerrorDEIM}
	\end{figure} \end{center}

To construct our ROM, we sample at four parameter values $ \bhp^{(i)}$ for $i=1,2,3,4$ (see the leading four rows in Table \ref{tab1}).
We calculate the corresponding projection matrices $ \bV_1, \bV_2, \bV_3, \bV_4, \bW_1, \bW_2, \bW_3, $ and $\bW_4 $ using Theorem \ref{thm:hess} 
that guarantee interpolation and sensitivity matching at the frequency interpolation points (generated via IRKA once again) and tangential directions corresponding to each parameter value sampled. To maintain symmetry in $\bEr$ and $\bAr$,
we concatenate all of the projection matrices and consider a one-sided projection, i.e., 
	$ \bV = [ \bV_1 \ \bV_2 \ \bV_3 \  \bV_4 \bW_1 \ \bW_2 \ \bW_3 \   \bW_4] $ 
and 	$ \bW = \bV $.
We truncate the basis (using SVD) and obtain a ROM with dimension $ r = 20 $.
	\begin{table}
	\centering
	\begin{tabular}{|l||ccc|} 
	\hline\noalign{\smallskip}
			& viscosity	 	& source center 	& source reach \\
	\noalign{\smallskip} \hline \noalign{\smallskip}
	$ \bhp^{(1)} $	& 0.1 			& $ (0.25,0.8) $ 		& 1 \\
	$ \bhp^{(2)} $	& 1 			& $ (0,0) $			& 9 \\
	$ \bhp^{(3)} $	& $ e^{-3} $		& $ (1,1) $			& 4 \\
	$ \bhp^{(4)} $	& $ e^{-3} $		& $ (-0.5,-1) $ 		& 1 \\ \hline 
        $ \bp^{(1)} $	& 0.0529 		& $ (0.975,0.9275)	$	& 1.6636 \\
	$ \bp^{(2)} $	& 0.2392 		& $ (0.6914,0.3149) $ 	& 3.6730 \\
	$ \bp^{(3)} $	& 0.1261 		& $ (-0.7224,-0.7623) $ 	& 5.1100 \\
	$ \bp^{(4)} $	& 0.0754 		& $ (-0.3214,0.4988) $ 	& 2.816 \\
	\noalign{\smallskip}\hline
	\end{tabular}
					\caption{Advection-diffusion model. Parameter values.}
						\label{tab1}
	\end{table}

To illustrate the accuracy of the reduced model, we test it for two different inputs sets  (see Table \ref{tab2})
and for four different parameter samples $ \bp^{(i)}$ for $i=1,2,3,4$ (see Table \ref{tab1}, rows 5--9) that were not part of the sampling set. We show the results, the full-order   and reduced-order outputs
in Figures \ref{ADRpar_i2} and \ref{ADRpar_i1}  for two different inputs (see Table \ref{tab2}). Both figures show that for each input selection (neither of which entered into our transfer function-based model reduction process), the parametric reduced bilinear model provides a high-quality approximation, only showing slight variations at the parameter values that were not sampled.
	\begin{table}
	\centering
	\begin{tabular}{|c||cccc|} 
	\hline\noalign{\smallskip}
			& $ u_1 $ 	& $ u_2 $ 	& $ u_3 $	& $ u_4 $ \\
	\noalign{\smallskip} \hline \noalign{\smallskip}
	Input 1	& $ \sin t $	& $ \cos t $	& -1		& 0.5  \\
	Input 2	& 0.5 		& 0.25	& 1		& -1 \\
	\noalign{\smallskip}\hline
	\end{tabular}
		\caption{Advection-diffusion model. Input values.}
				\label{tab2}
	\end{table}

	 \begin{figure} [h!]
	\includegraphics[width=0.45\textwidth]{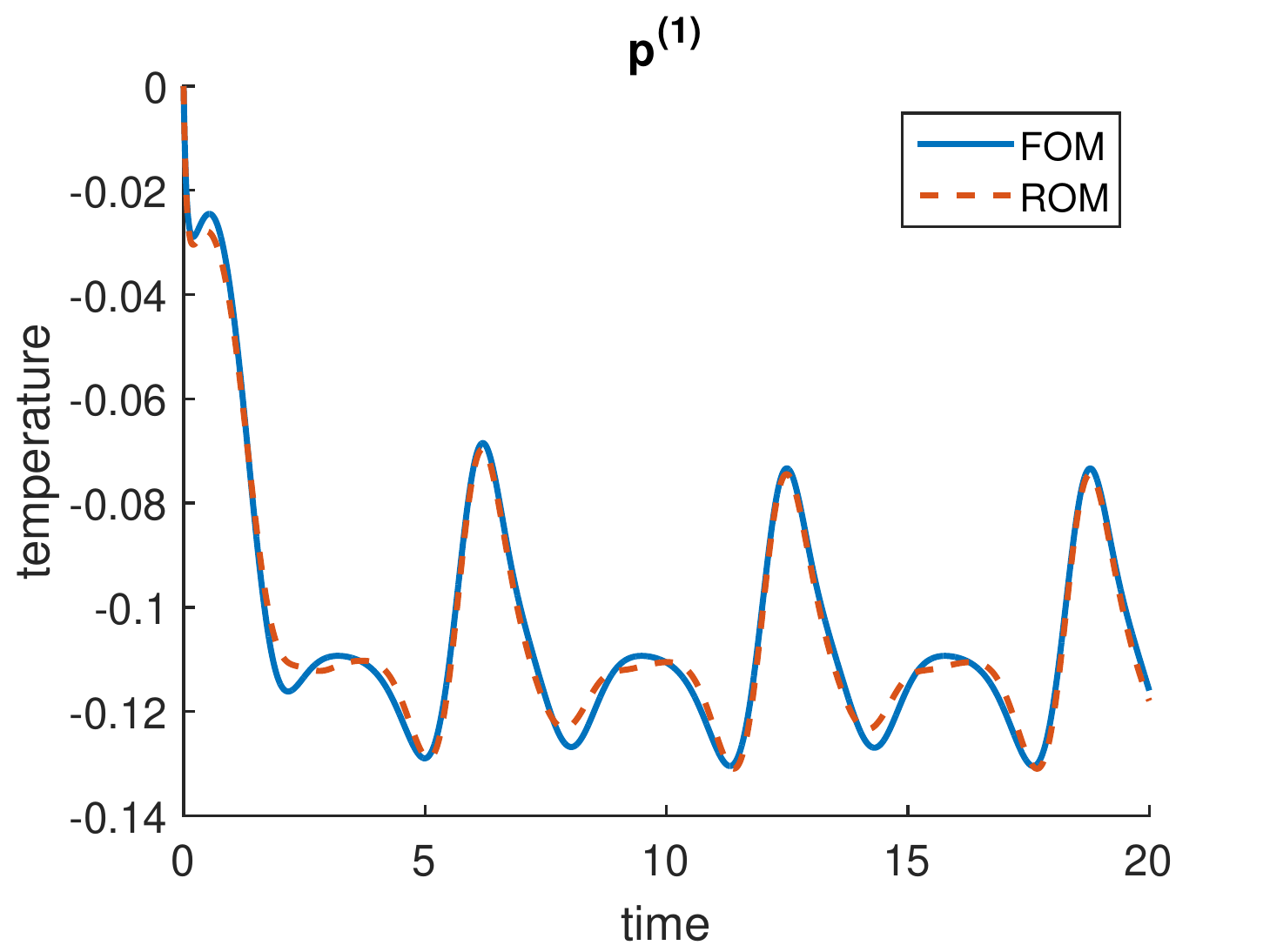}  \
	\includegraphics[width=0.45\textwidth]{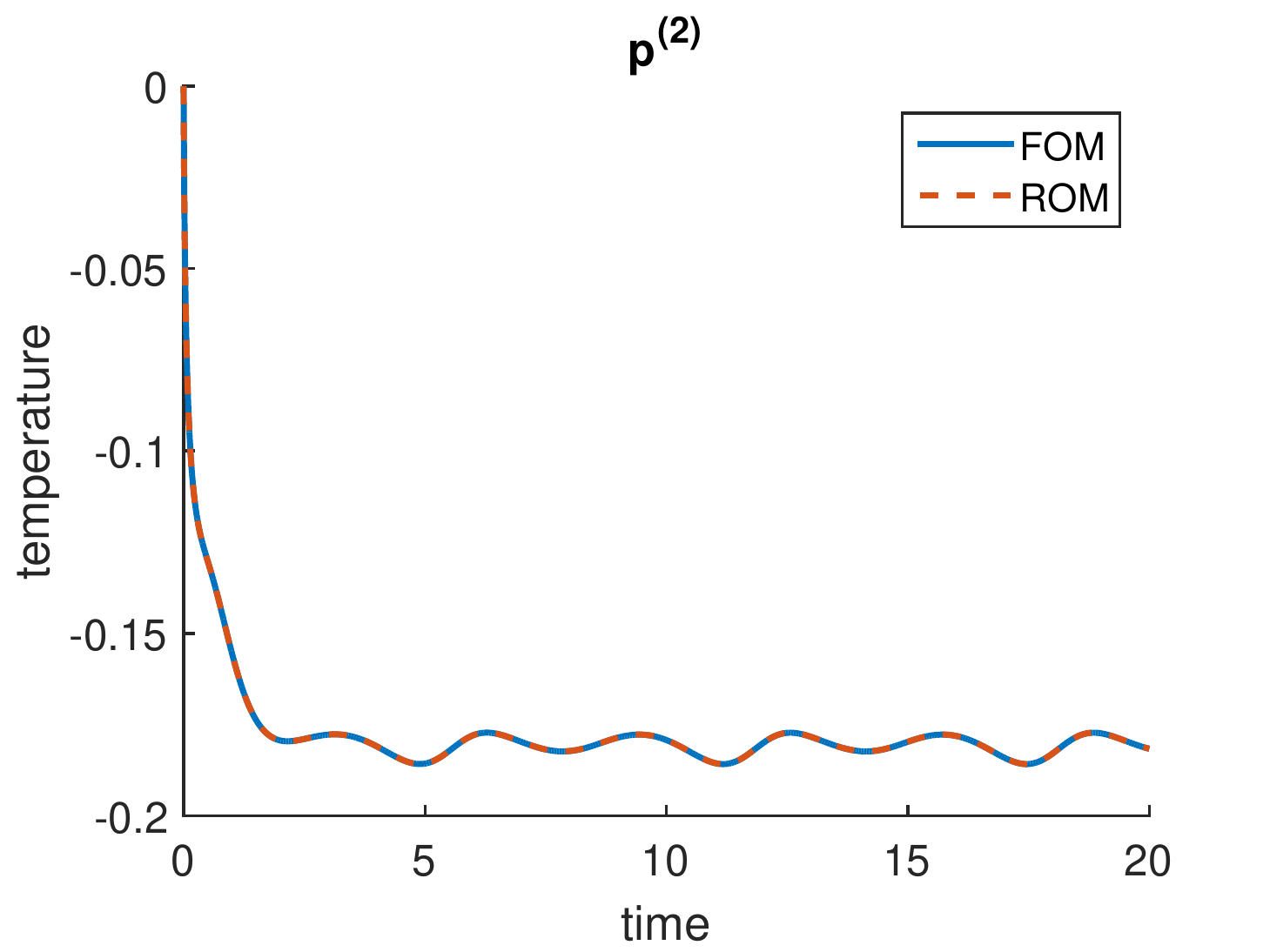}  \\
	\includegraphics[width=0.45\textwidth]{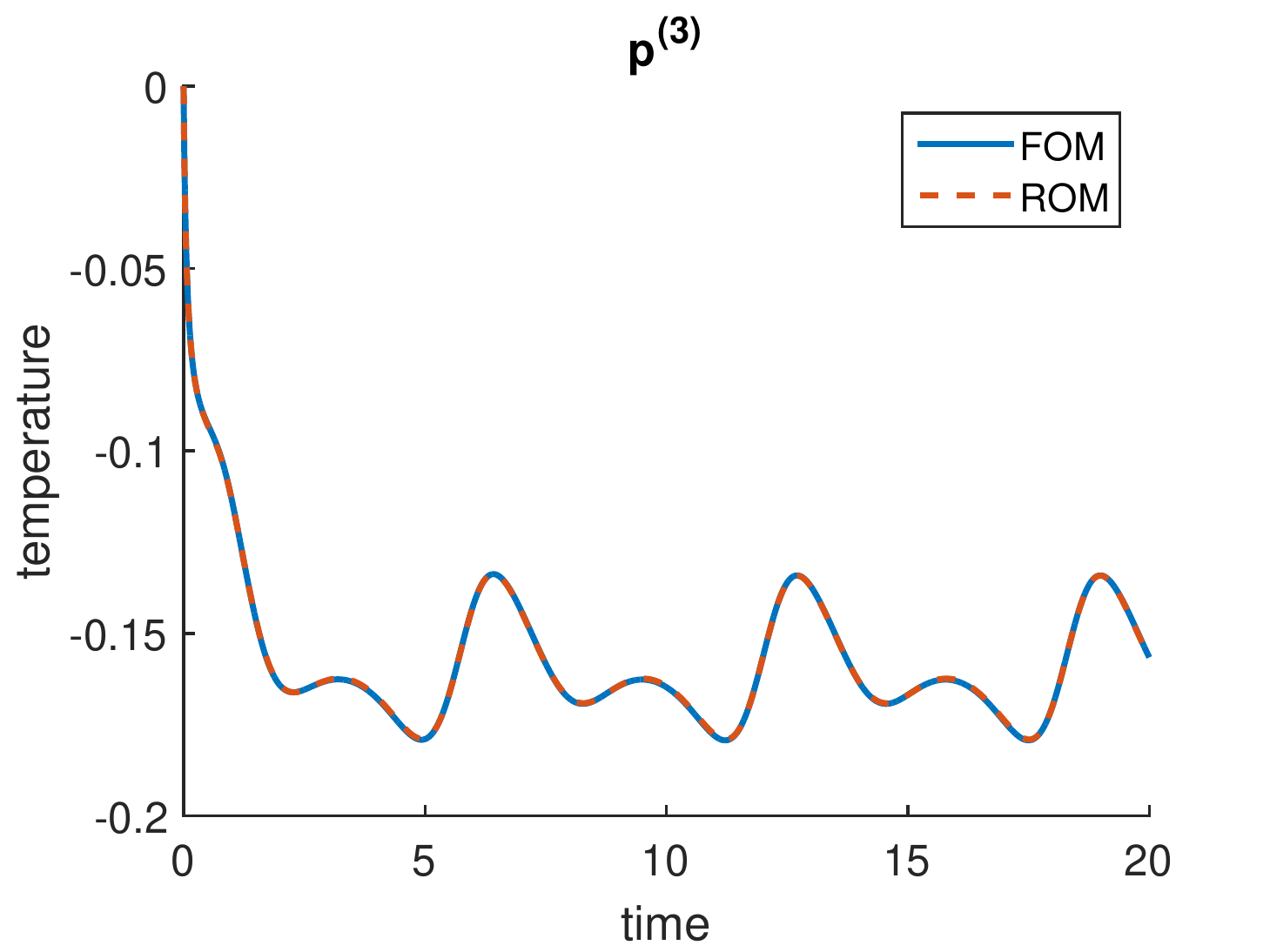}  \
	\includegraphics[width=0.45\textwidth]{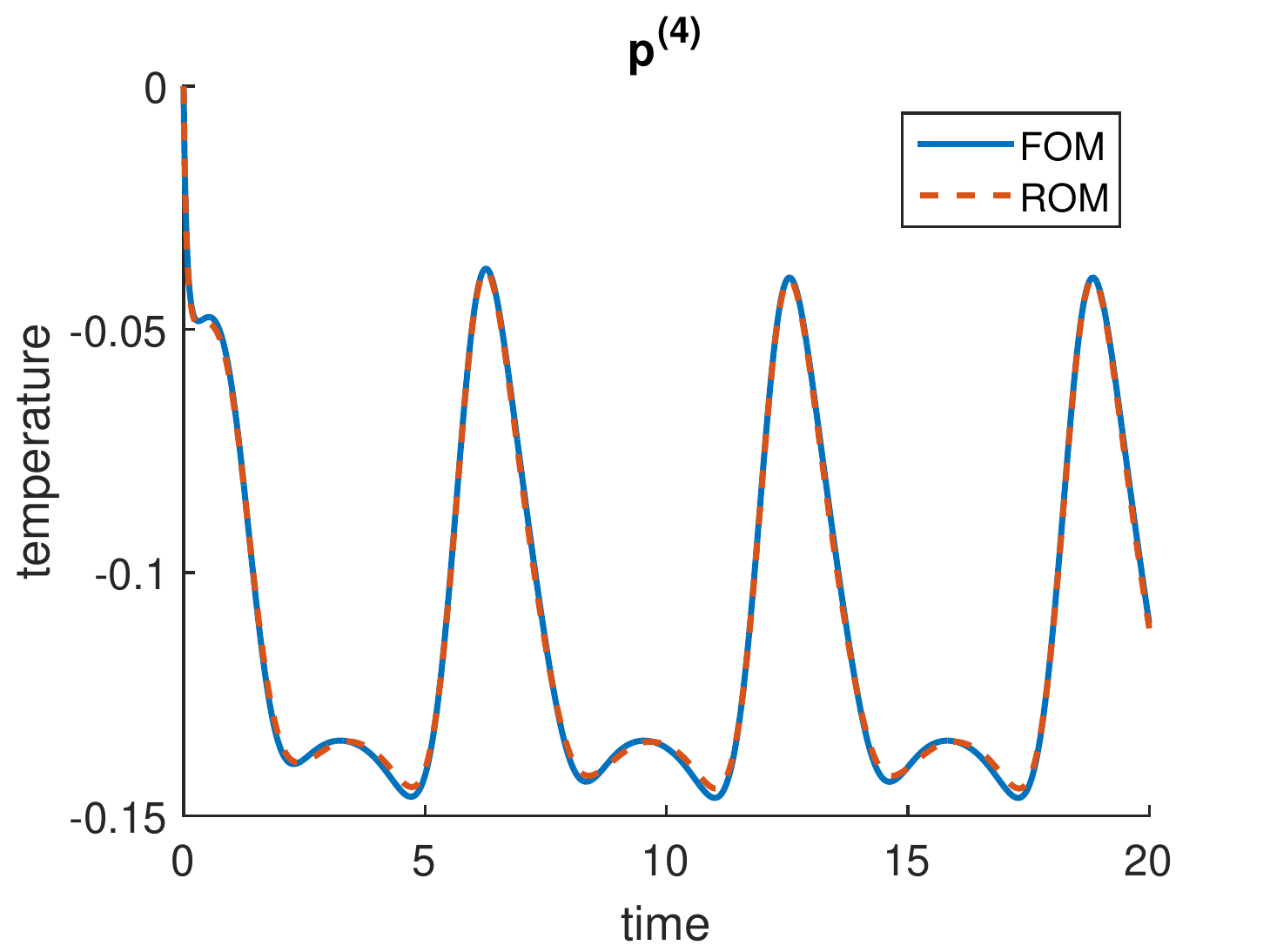}  \

	\caption{Solution of the full (FOM) and reduced order model (ROM) corresponding to non-sampled parameter values (see values in Table \ref{tab1}) for Input 1 with entires 
	$ u_1 (t) = \sin t $, $ u_2 (t) = \cos t $, $ u_3 (t) = -1 $, $ u_4 (t) = 0.5 $.}
	\label{ADRpar_i2}
	\end{figure} 

	 \begin{figure} [h!]
	\includegraphics[width=0.45\textwidth]{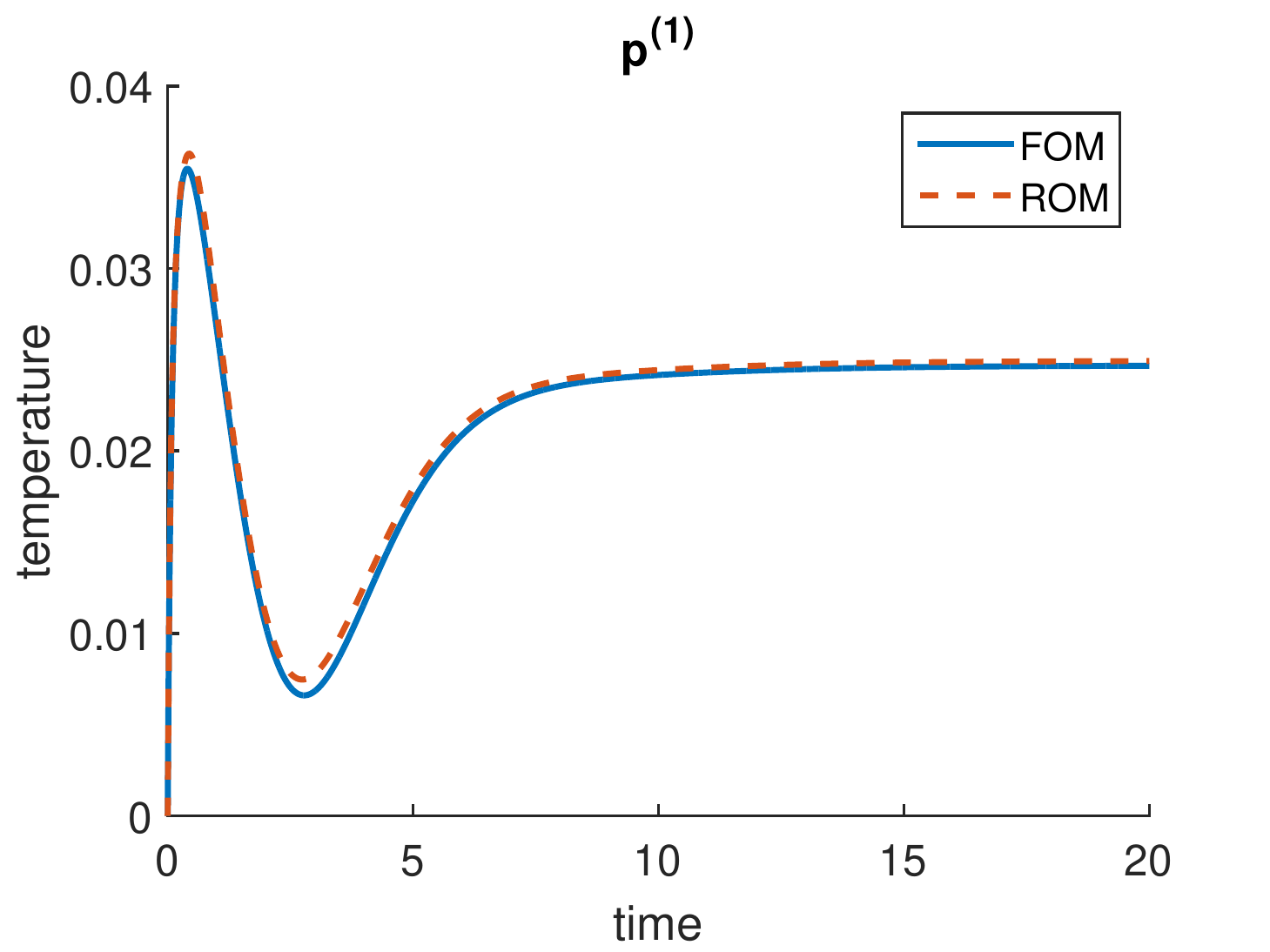}  \
	\includegraphics[width=0.45\textwidth]{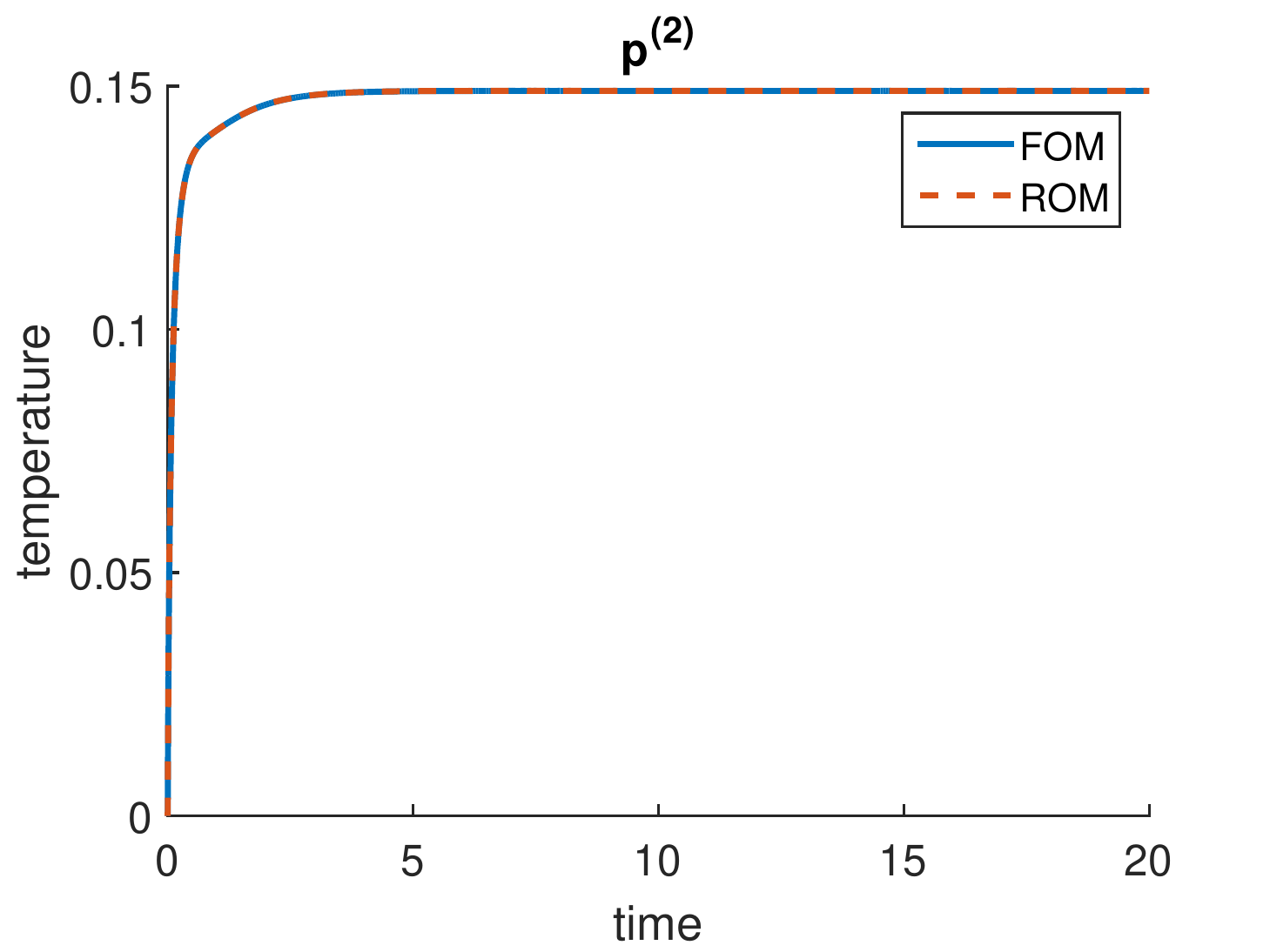}  \\
	\includegraphics[width=0.45\textwidth]{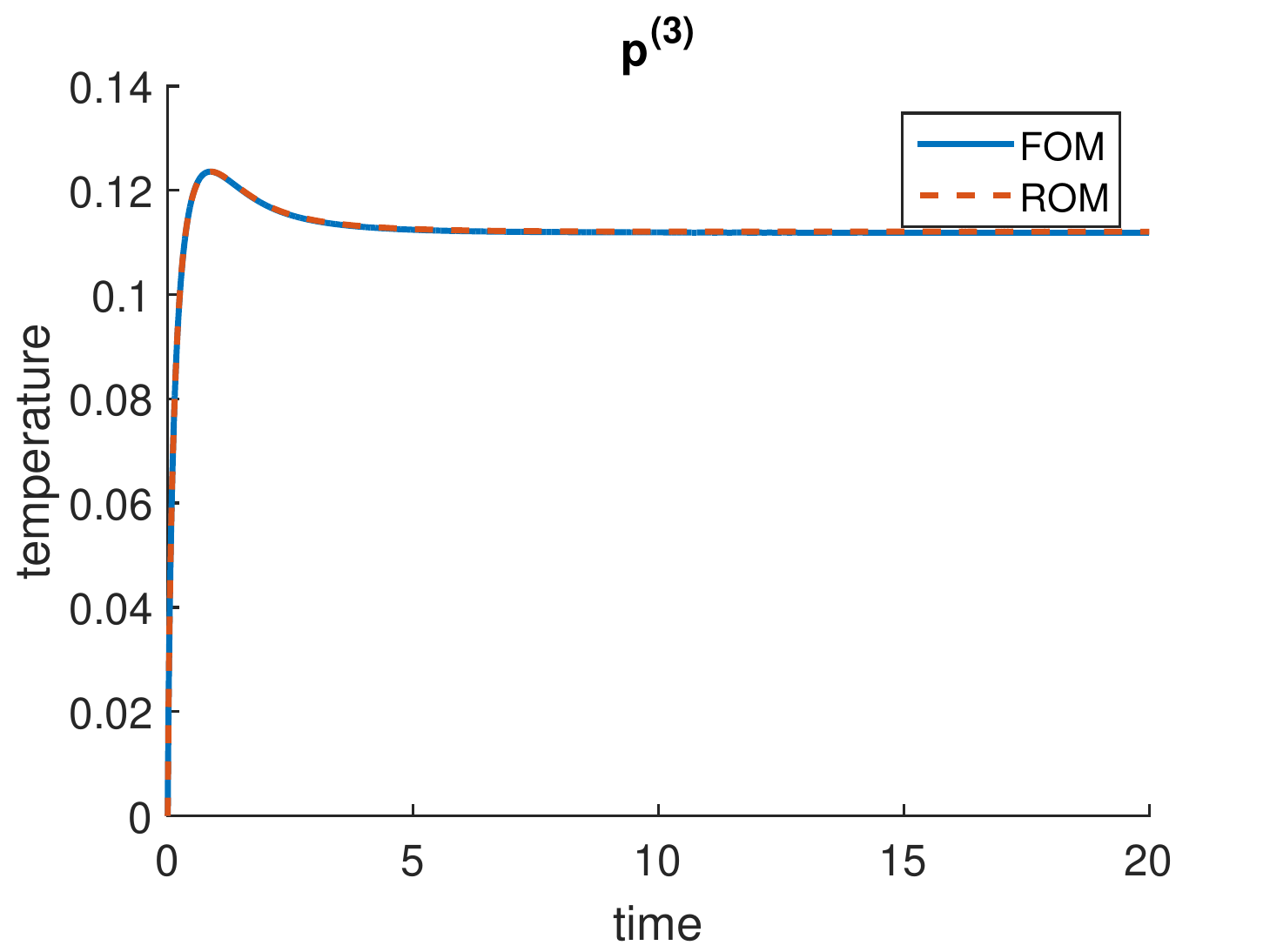}  \ 
	\includegraphics[width=0.45\textwidth]{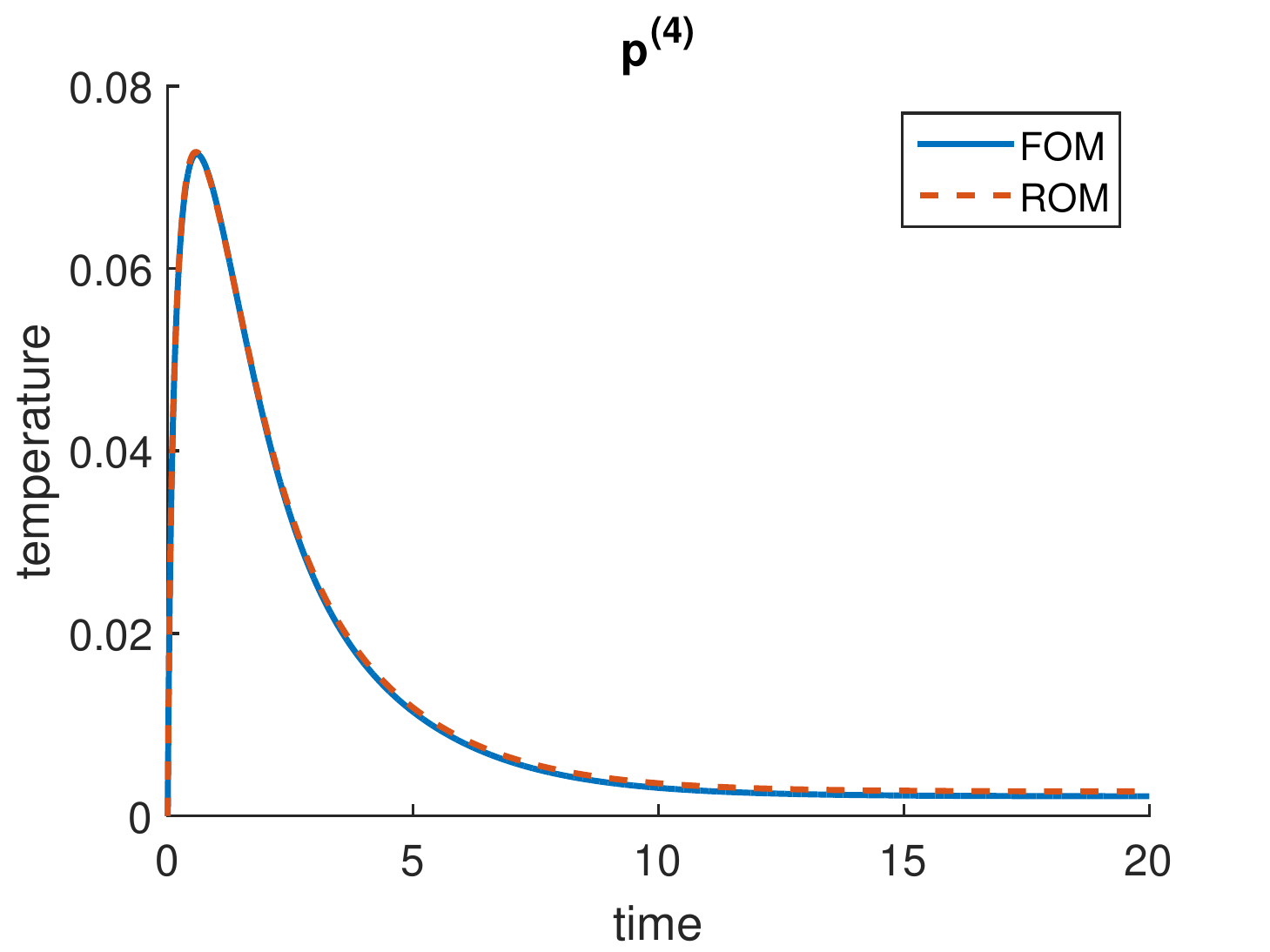}  \

	\caption{Solution of the full (FOM) and reduced order model (ROM) corresponding to non-sampled parameter values (see values in Table \ref{tab1}) for for Input 1 with entires 
	$ u_1 (t) = 0.5 $, $ u_2 (t) = 0.25 $, $ u_3 (t) = 1 $, $ u_4 (t) = -1 $.}
	\label{ADRpar_i1}
	\end{figure}

As in the previous example, to ensure a fair comparison, we performed an exhaustive search via $10^4$ uniform random samples in our full parameter domain (except that we fixed the fourth parameter entry at $ p_4 = 5 $ so that we can present the results with a $3$-dimensional plot) for both input selections from Table~\ref{tab2}. Out of these $10^4$ parameter selections, in Figure~\ref{ADRworst_i2}, we display in the left-plot the relative errors at every sampling point and in the right-plot, the outputs for the \emph{worst} performance of the  reduced model for Input $1$. Note that even for the worst parameter sample, the parametric reduced model still provides an accurate approximation with a relative $L_2$ error of \textcolor{black}{$4.79 \times 10^{-2}$}. We repeat the procedure for Input 2 in Figure~\ref{ADRworst_i1} and obtain similar results. 
	\begin{center} \begin{figure} [h!]

	\includegraphics[width=0.45\textwidth]{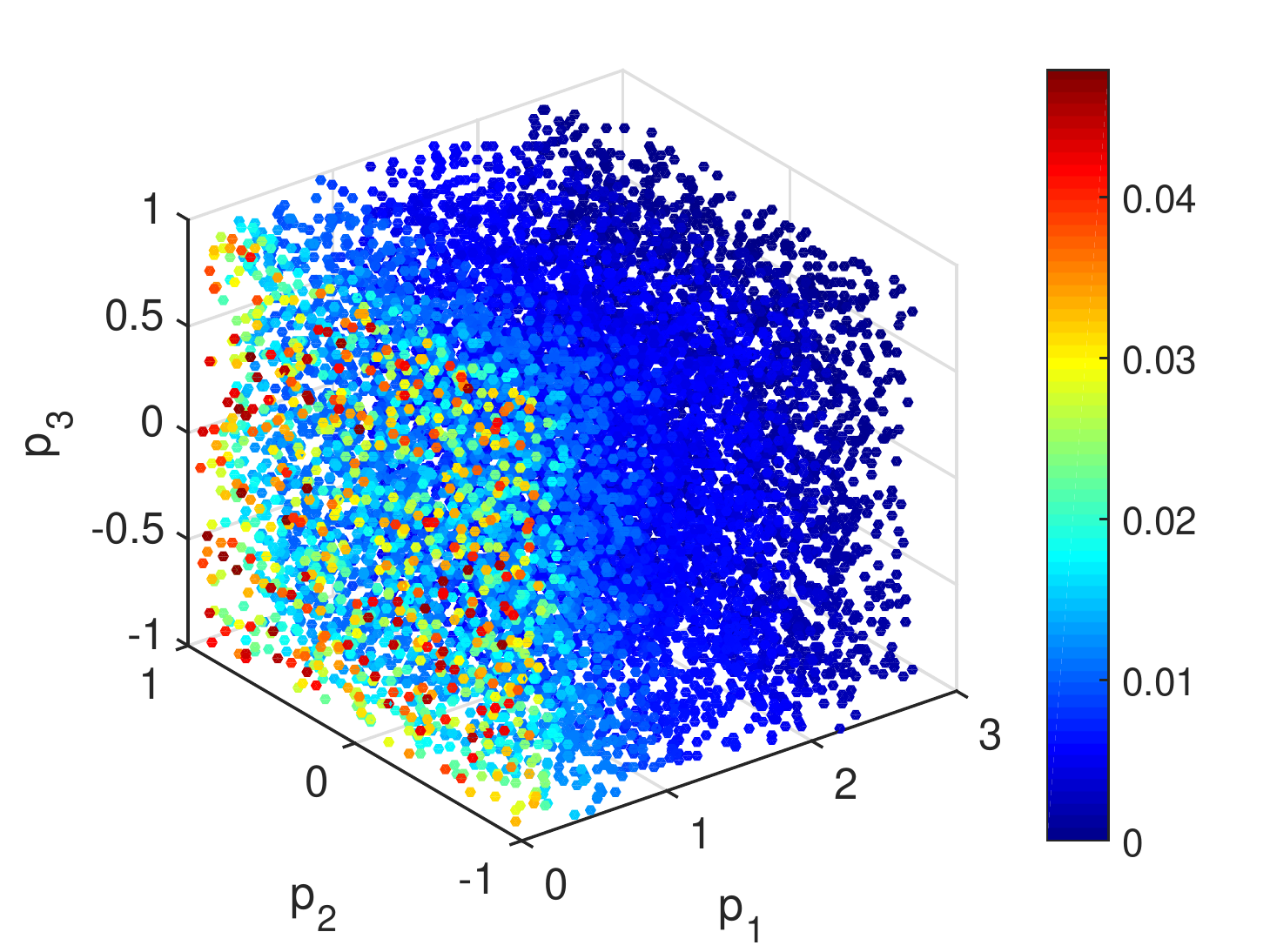} \ 
	\includegraphics[width=0.45\textwidth]{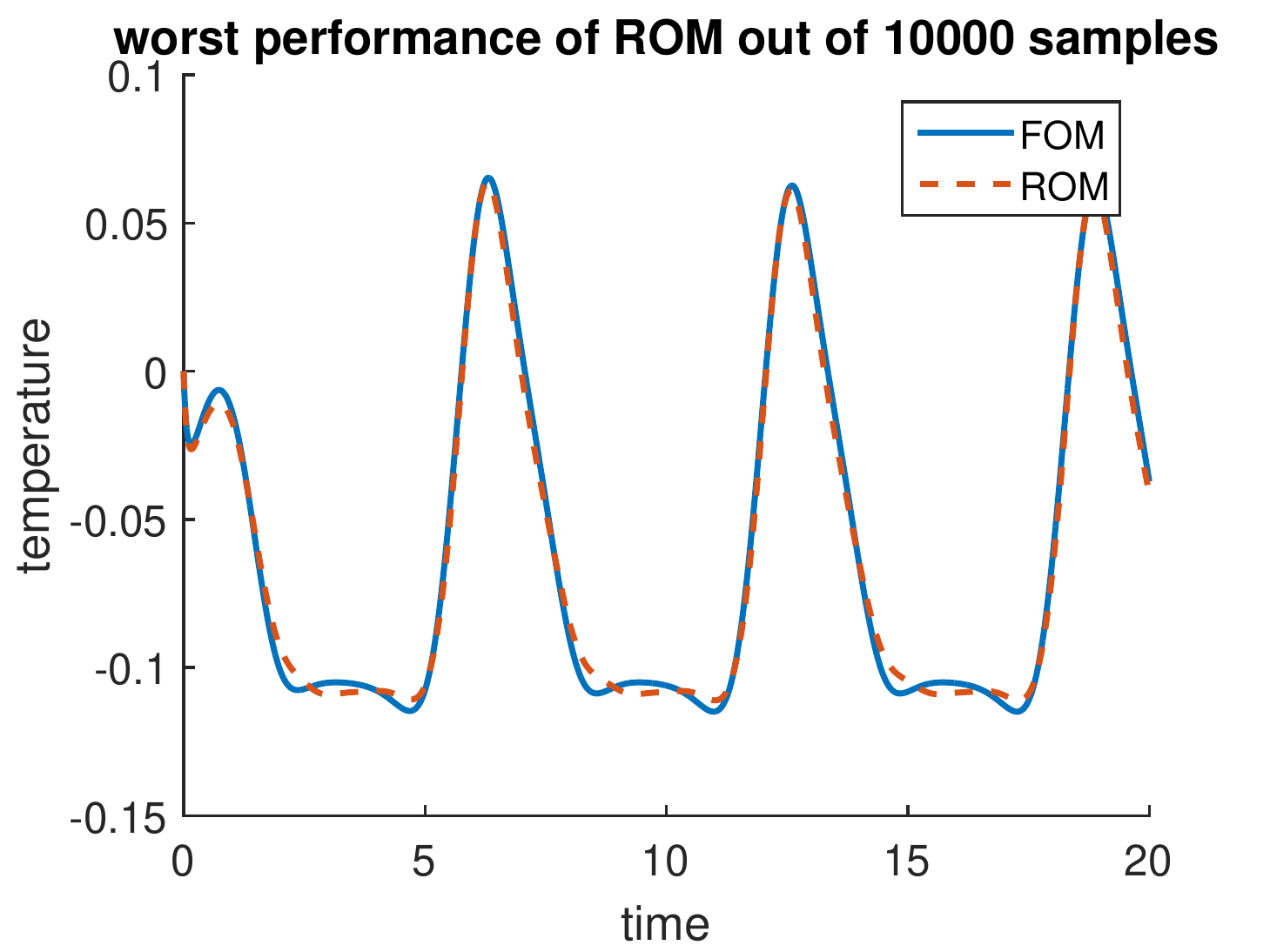}  	
	
	\caption{For Input 1: 
	on the left, the relative $ L_2 $ output error;
	on the right, solution of the full order model and the reduced order model corresponding to the highest relative $ L_2 $ error.}
	\label{ADRworst_i2}
	\end{figure} \end{center}

	\begin{center} \begin{figure} [h!]

	\includegraphics[width=0.45\textwidth]{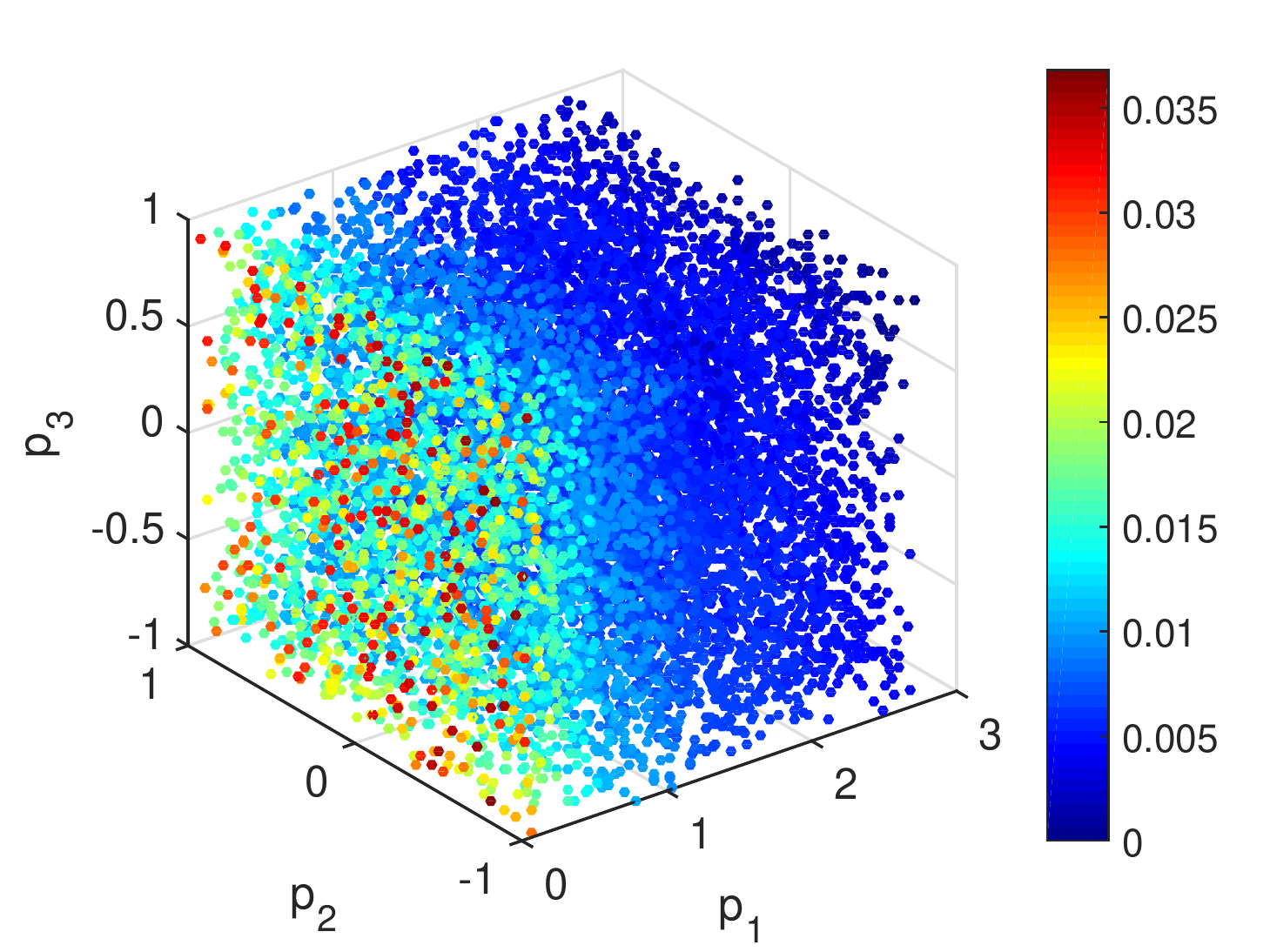} \ 
	\includegraphics[width=0.45\textwidth]{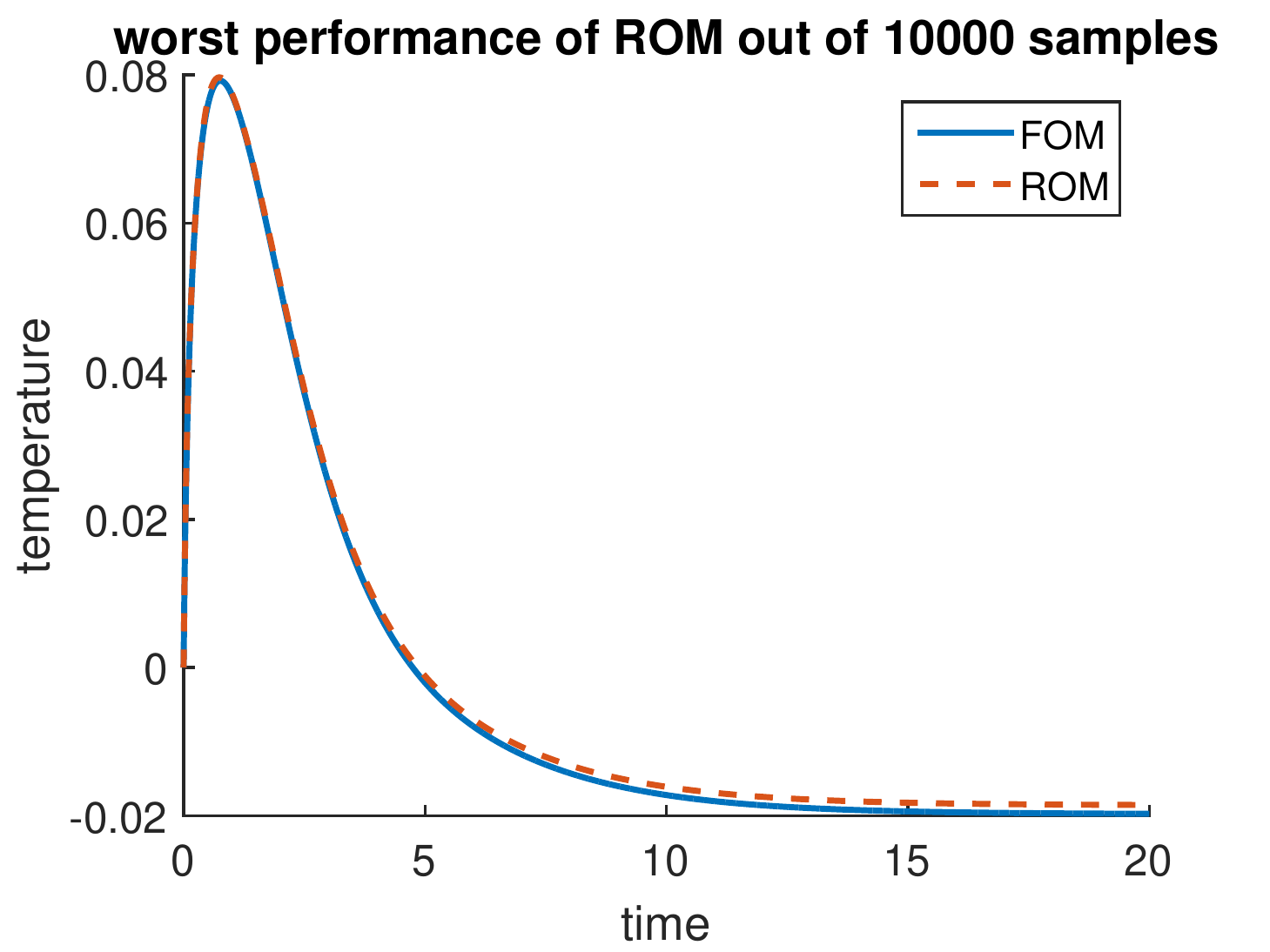}  	
	
	\caption{For Input 2: 
	on the left, the relative $ L_2 $ output error;
	on the right, solution of the full order model and the reduced order model corresponding to the highest relative $ L_2 $ error.}
	\label{ADRworst_i1}
	\end{figure} \end{center}


\section{Conclusions and Future Work} \label{sec:conc}

In this paper, we presented conditions that ensure Hermite interpolation conditions for parametric bilinear systems.  
These conditions also ensure that parametric directional derivatives of the reduced-order transfer functions match the full-order transfer function at given interpolation points and directions.  
We demonstrate the quality of our model reduction using two examples, one a well-known benchmark and the other an interesting advection-diffusion equation.  
The performance was very good, and we emphasize that no effort was made to select the sample points in parameter space.  
In fact, this approach is agnostic to the parameter choices and can be easily embedded in well-known parameter selection schemes.  
The next natural steps are to test this algorithm with different schemes and more challenging problems.  
We also intend to extend this approach to parametric quadratic nonlinear systems.
 
\bibliographystyle{plain}
\bibliography{ArXiv_draft}	

\end{document}